\documentclass[12pt]{amsart}
\usepackage[utf8]{inputenc}
\usepackage[shortlabels]{enumitem}
\setlist[enumerate]{itemsep=0.5ex}
\usepackage{url}
\usepackage{hyperref}
\usepackage{pdfpages}
\usepackage[export]{adjustbox}

\usepackage{amsbsy}
\usepackage{amsfonts}
\usepackage{amsgen}
\usepackage{amsmath}
\usepackage{amsopn}
\usepackage{amssymb}
\usepackage{amstext}
\usepackage{amsthm}
\usepackage{amsxtra}
\usepackage{mathtools}
\usepackage{thmtools}
\usepackage{dsfont}
\usepackage{cleveref}

\theoremstyle{definition}
\newtheorem{definition}{Definition}[section]
\newtheorem{example}[definition]{Example}
\newtheorem{remark}[definition]{Remark}
\newtheorem{question}[definition]{Question}

\theoremstyle{plain}
\newtheorem{proposition}[definition]{Proposition}
\newtheorem{lemma}[definition]{Lemma}
\newtheorem{theorem}[definition]{Theorem}
\newtheorem{corollary}[definition]{Corollary}

\newcommand{\N}{\mathbb{N}}

\newcommand{\C}{\mathbb{C}}

\newcommand{\A}{\mathcal{A}}
\newcommand{\B}{\mathcal{B}}
\newcommand{\OO}{\mathcal{O}}
\newcommand{\PP}{\mathcal{P}}

\newcommand{\abs}[1]{\left\lvert#1\right\rvert}

\DeclareMathOperator{\id}{id}
\DeclareMathOperator{\Spec}{Spec}
\DeclareMathOperator{\Aut}{Aut}

\DeclareMathOperator{\op}{op}
\DeclareMathOperator{\Ban}{Ban}
\DeclareMathOperator{\Bic}{Bic}
\DeclareMathOperator{\GH}{GH}

\newcommand*{\AutBanica}{\Aut^+_{\Ban}}
\newcommand*{\AutBichon}{\Aut^+_{\Bic}}
\newcommand*{\AutGHBichon}{\Aut^+_{\GH,\Bic}}
\newcommand*{\Adj}{A_{\Gamma}}

\begin{document}

\title{Quantum Automorphism Groups of Hypergraphs}
\author{Nicolas Faroß}
\address{Saarland University, Fachbereich Mathematik, Postfach 151150, 66041 Saarbr\"ucken, Germany}
\email{faross@math.uni-sb.de}
\date{\today}
\thanks{The author thanks his supervisor Moritz Weber for many helpful comments and suggestions.
This article is part of the author's PhD thesis.
This work is a contribution to the SFB-TRR 195.}

\begin{abstract}
\noindent We introduce a quantum automorphism group for 
hypergraphs, which turns out to generalize the quantum 
automorphism group of Bichon for classical graphs. Further, 
we show that our quantum automorphism group acts 
on hypergraph $C^*$-algebras as recently defined. In particular, 
this action generalizes the one on graph $C^*$-algebras 
by Schmidt-Weber in 2018.
\end{abstract}

\maketitle

\section{Introduction}\label{sec:introduction}

\noindent In~\cite{wang98}, Wang introduced quantum automorphism groups of finite space.
These are compact quantum groups in the sense of Woronowicz~\cite{woronowicz87, woronowicz91}
and they can be used to describe symmetries in the setting of $C^*$-algebras.
A particular example is the quantum permutation group $S_n^+$, 
which generalizes the classical symmetry group of $n$ points $S_n$. 
It can be defined by the universal unital $C^*$-algebra 
\[
  C(S_n^+) := C^*( u_{ij} \ | \ \text{${(u_{ij})}$ is a magic unitary}),
\]
where a $n \times n$ matrix $u := (u_{ij})$ is called a magic unitary, if 
\[
  u_{ij}^2 = u_{ij}^* = u_{ij},
  \qquad 
  \sum_{k=1}^n u_{ik} = \sum_{k=1}^n u_{kj} = 1 
  \qquad 
  (1 \leq i, j \leq k).
\]
Magic unitaries are also called quantum permutations since magic unitaries with entries in $\C$ are exactly classical permutation matrices.

In the same setting, Bichon~\cite{bichon03} and Banica~\cite{banica05} introduced two versions
of quantum automorphism groups of finite graphs.
These quantum groups generalize the classical automorphism group of a graph by imposing the additional 
relation $A_\Gamma u = u A_\Gamma$ on a magic unitary $u$. Here, $A_\Gamma$ denotes the adjacency matrix of a graph $\Gamma$, which
requires the magic unitary $u$ to respect the graph structure.
Quantum automorphism groups of graphs provide a large class of examples
of compact quantum groups and have for example been studied in~\cite{chassaniol19, schmidt20, levandovskyy22, bruyn23}.
In particular, these quantum automorphism groups have been further generalized to different structures like 
multigraphs~\cite{goswami23}, Hadamard matrices~\cite{gromada22} and quantum graphs~\cite{brannan20, brannan22}.

\subsection{Quantum automorphism groups of hypergraphs}

In this paper, we present a definition of a quantum automorphism group for hypergraphs.
Hypergraphs generalize classical graphs by allowing an edge to connect not only two but an arbitrary number
of vertices. This makes hypergraphs quite general and gives them many applications in 
discrete mathematics and computer science~\cite{ausiello17, gallo93}.
In particular, it is possible to form the dual of a hypergraph by interchanging its vertices 
and edges, which is in general not possible for classical graphs.
See~\cite{berge84} for further information on hypergraphs.

In the following, a hypergraph $\Gamma := (V, E)$ is given 
by a finite set of vertices $V$, a finite set of edges $E$ and 
two maps 
\[
  s \colon E \to \PP(V),
  \quad 
  r \colon E \to \PP(V).
\]
An edge $e \in E$ can be depicted by an arrow 
from the set of source vertices $s(e)$ to the set of range vertices $r(e)$.
Thus, we can view classical directed edges as hyperedges with $\abs{s(e)} = \abs{r(e)} = 1$.
Note that we consider directed hypergraphs which can 
also have empty edges and multi-edges. In this setting, 
our quantum automorphism group $\Aut^+(\Gamma)$ of a hypergraph $\Gamma$
is given by the following compact matrix quantum group.

\begingroup
\def\thedefinition{1}
\begin{definition}[\Cref{def:hypergraph-qaut}]
  Let $\Gamma := (V, E)$ be a hypergraph and $\A$ the universal unital $C^*$-algebra 
  with generators $u_{vw}$ for all $v, w \in V$ and $u_{ef}$ for all $e,f \in E$
  such that
  \begin{enumerate}
    \item $u_V := {(u_{v w})}_{v, w \in V}$ and $u_E := {(u_{e f})}_{e, f \in E}$ are magic unitaries,
    \item $A_s u_E = u_V A_s$ and $A_r u_E = u_V A_r$, where $A_r, A_s \in \C^{V \times E}$ are defined by 
    \[
      {(A_s)}_{ve} := \begin{cases}
        1 & \text{if $v \in s(e)$,} \\
        0 & \text{otherwise,}
      \end{cases}
      \qquad
      {(A_r)}_{ve} := \begin{cases}
        1 & \text{if $v \in r(e)$,} \\
        0 & \text{otherwise,}
      \end{cases}
    \]
    for all $v \in V$ and $e \in E$.
  \end{enumerate}
  Then $\Aut^+(\Gamma) := (\A, u_V \oplus u_E)$ is the \textit{quantum automorphism group of the hypergraph $\Gamma$}.
\end{definition}
\endgroup

Intuitively, $\Aut^+(\Gamma)$ is given by a quantum permutation $u_V$ on the 
vertices and quantum permutation $u_E$ on the edges, which are compatible
by intertwining the incidence matrices $A_s$ and $A_r$. If $u_V$ and $u_E$ are classical 
permutation matrices, then this definition gives exactly the classical 
automorphism group of a hypergraph, see \Cref{sec:classic-hypergraph-aut}.

Note that in contrast to the quantum automorphism groups of Bichon and Banica, 
we include a second magic unitary for the edges. This is 
necessary to capture quantum symmetries between multi-edges,
which are allowed in our definition of hypergraph.
See~for example \Cref{sec:example-max-qsym} for the quantum symmetries of a concrete family of hypergraphs with multi-edges.
However, if a hypergraph $\Gamma$ or its dual $\Gamma^*$ have no multi-edges, then our definition 
reduces to only one magic unitary.

\begingroup
\def\thetheorem{1}
\begin{theorem}[\Cref{corr:subgroup-SV}, \Cref{corr:subgroup-SE}]
Let $\Gamma:= (V, E)$ be a hypergraph.
\begin{enumerate}
  \item If $\Gamma$ has no multi-edges, then $\Aut^+(\Gamma) \subseteq S_V^+$.
  \item If $\Gamma^*$ has no multi-edges, then $\Aut^+(\Gamma) \subseteq S_E^+$.
\end{enumerate}
\end{theorem}
\endgroup

Since hypergraphs are a generalization of classical graphs and multigraphs,
it is natural to ask how our quantum automorphism group
relates to the quantum automorphism groups of Bichon~\cite{bichon03}, Banica~\cite{banica05} and Goswami-Hossain~\cite{goswami23}.
Denote with $\AutBichon(\Gamma)$ the quantum automorphism group of Bichon and with
$\AutGHBichon(\Gamma)$ the quantum automorphism group of Goswami-Hossain in the sense of Bichon.
Then the following theorem shows, that we obtain both quantum groups as a special case
when encoding classical graphs and multigraphs as hypergraphs.

\begingroup
\def\thetheorem{2}
\begin{theorem}[\Cref{thm:directed-graph-bichon}, \Cref{thm:simple-graph-bichon}, \Cref{thm:multigraph-GH}] 
\leavevmode
\begin{enumerate}
\item Let $\Gamma := (V, E)$ be a directed graph as in \Cref{def:directed-graph} and define 
\[  
  s(v, w) := \{ v \},
  \quad 
  r(v, w) := \{ w \}
  \quad 
  \forall (v, w) \in E.
\]
Then $\Gamma$ is a hypergraph with $\Aut^+(\Gamma) = \AutBichon(\Gamma)$.
\item Let $\Gamma := (V, E)$ be a simple graph as in \Cref{def:simple-graph} and define 
\[  
  s(\{v, w\}) := \{v, w\}, 
  \quad 
  r(\{v, w\}) := \{v, w\}
  \quad 
  \forall \{v, w\} \in E.
\]
Then $\Gamma$ is a hypergraph with  $\Aut^+(\Gamma) = \AutBichon(\Gamma)$.
\item Let $\Gamma := (V, E)$ be a multigraph as in \Cref{def:multi-graph} 
with source map $s'$ and range map $r'$. Define 
\[  
  s(e) := \{ s'(e) \}, 
  \quad
  r(e) := \{ r'(e) \}
  \quad 
  \forall e \in E.
\]
Then $\Gamma$ is a hypergraph with $\Aut^+(\Gamma) = \AutGHBichon(\Gamma)$.
\end{enumerate}
\end{theorem}
\endgroup

Note that the previous theorem can also be used to construct many concrete examples of quantum automorphism groups of hypergraphs.

\subsection{Quantum symmetries of hypergraph $C^*$-algebras}

Related to quantum automorphism groups of classical graphs is the study of quantum symmetries
of graph $C^*$-algebras in~\cite{schmidt18,joardar18}. 
These $C^*$-algebras are defined in terms of an underlying graph 
and have been studied since the 1980's. They include many examples like matrix algebras, continuous functions on the circle or the 
Cuntz algebras, see~\cite{raeburn05} for more details. 
Recently, Trieb-Weber-Zenner~\cite{trieb24} introduced hypergraph $C^*$-algebras, 
which generalize graph $C^*$-algebras to the setting hypergraphs. 
This new class includes all graph $C^*$-algebras but also new examples of non-nuclear $C^*$-algebras.
See also the recent work by Schäfer-Weber~\cite{schaefer24} which characterizes the nuclearity of hypergraph $C^*$-algebras 
in terms of minors of the underlying hypergraph.

In\cite{schmidt18}, Schmidt-Weber showed that Banica's quantum automorphism group 
acts maximally on the corresponding graph $C^*$-algebra.
We generalize this result to 
hypergraphs by showing that our quantum automorphism group 
$\Aut^+(\Gamma)$ acts on the corresponding hypergraph $C^*$-algebra $C^*(\Gamma)$. 
As in the case of graph $C^*$-algebras, hypergraph $C^*$-algebras are generated by a family of projections ${\{ p_v \}}_{v \in V}$ and a family of partial isometries ${\{ s_e \}}_{e \in E}$, see \Cref{def:hypergraph-Cstar-alg}.
Our action is then given by permuting these generators using the 
magic unitaries $u_V$ and $u_E$, i.e.
\[
  \alpha(p_v) = \sum_{w \in V} p_w \otimes u_{wv},
  \qquad
  \alpha(s_e) = \sum_{f \in E} s_f \otimes u_{fe}
  \qquad
  \forall v \in V, \, e \in E.
\]

It turns out that this action equals the one of Schmidt-Weber in the 
case of classical graphs, but our quantum automorphism group is no longer maximal with respect to it. However, we obtain maximality 
under the additional assumption that $\Aut^+(\Gamma)$ also acts on 
$C^*(\Gamma')$, where $\Gamma'$ is obtained by inverting all edge directions and 
by interchanging the vertices and edges of $\Gamma$. 

\begingroup
\def\thetheorem{3}
\begin{theorem}[\Cref{thm:action-existence}, \Cref{thm:action-maximal}]
Let $\Gamma := (V, E)$ be a hypergraph
and $\Gamma' := {(\Gamma^{*})}^{\op}$.
Then 
$\Aut^+(\Gamma)$ is the largest compact matrix quantum group 
which acts faithfully
on both $C^*(\Gamma)$ and $C^*(\Gamma')$ via 
\begin{align*}
  \alpha_1 \colon C^*(\Gamma) &\to C^*(\Gamma) \otimes C(\Aut^+(\Gamma)), \\
  \alpha_2 \colon C^*(\Gamma') &\to C^*(\Gamma') \otimes C(\Aut^+(\Gamma))
\end{align*}
with 
\begin{alignat*}{3}
  \alpha_1(p_v) &= \sum_{w \in V} p_w \otimes u_{wv}
  \quad
  \forall v \in V,
  &\qquad
  \alpha_1(s_e) &= \sum_{f \in E} s_f \otimes u_{fe}
  \quad
  \forall e \in E, \\
  \alpha_2(p_e) &= \sum_{f \in E} p_f \otimes u_{fe}
  \quad
  \forall e \in E,
  &\qquad
  \alpha_2(s_v) &= \sum_{w \in V} s_w \otimes u_{wv}
  \quad
  \forall v \in V.
\end{alignat*}
\end{theorem}
\endgroup

\subsection*{Overview of the article}

We begin in \Cref{sec:preliminaries} with some 
preliminaries about graphs, hypergraphs, quantum groups
as well as graph and hypergraph $C^*$-algebras.
Then we introduce our quantum automorphism groups of hypergraphs in \Cref{sec:qsym-hypergraphs} and 
explore some first properties. Further, we give examples of hypergraphs with maximal quantum symmetries and 
consider the case of hypergraphs without multi-edges.
In~\Cref{sec:qsym-graphs}, we compute the quantum automorphism groups of hypergraphs 
which are constructed from classical graphs and multigraphs. In particular, we show that these agree 
with the quantum automorphism group of Bichon or its multigraph version in the sense of Goswami-Hossain.
Finally, we construct an action on hypergraph $C^*$-algebras in \Cref{sec:action}
and present some remaining open questions in \Cref{sec:open-questions}.

\section{Preliminaries}\label{sec:preliminaries}

\subsection{Notations}

We begin by introducing some notations and conventions, 
which will be used throughout the rest of the paper.
In the following, many results will be formulated in the language $C^*$-algebras.
In particular, we will use universal $C^*$-algebras and tensor products of $C^*$-algebras.
For more information on these topics, we refer to~\cite{blackadar06}.

Let $X$ and $Y$ be sets. 
Then we denoted with $\PP(X)$ the power set of $X$ and with 
$X \sqcup Y$ the disjoint union of $X$ and $Y$. 
Assume $X$ and $Y$ are finite. Then $\C^X$ is the $C^*$-algebra of all complex-valued functions on $X$
with a basis ${\{ e_{i} \}}_{i \in X}$ given by indicator functions.
Similarly, $\C^{X \times Y}$
is the $C^*$-algebra of all complex-valued matrices with rows indexed by $X$ and 
columns indexed by $Y$. Here, the standard matrix units $e_{ij} \in \C^{X \times Y}$ 
are given by 
\[
  {(e_{ij})}_{k\ell} := \delta_{ik} \delta_{j\ell}
  \qquad
  \forall i, k \in X,\, j, \ell \in Y.
\]

If $\A$ and $\B$ are any $C^*$-algebras, then we denote with $\A \otimes \B$ the minimal tensor product
of $\A$ and $\B$.
Further, we identify elements $u \in \A \otimes \C^{X \times Y}$ with 
$\A$-valued matrices ${(u_{ij})}_{i\in X, j \in Y}$ via 
\[
  u = \sum_{i \in X} \sum_{j \in Y} u_{ij} \otimes e_{ij}.
\]

If $u \in \A \otimes \C^{X \times Y}$ and $v \in \A \otimes \C^{Y \times Y}$, then their direct sum 
$u \oplus v \in \A \otimes \C^{(X \sqcup Y) \times (X \sqcup Y)}$ is given by 
\[
  {(u \oplus v)}_{ij} := \begin{cases}
    u_{ij} & \text{if $i, j \in X$,} \\
    v_{ij} & \text{if $i, j \in Y$,} \\
    0      & \text{otherwise,} \\
  \end{cases}
\]
for all $i, j \in X \sqcup Y$.

Finally, assume $\A$ is unital. Then we identify scalar matrices $u \in \C^{X \times Y}$ with the $\A$-valued matrices $1 \otimes u \in \A \otimes \C^{X \times Y}$. 

\subsection{Graphs and hypergraphs}\label{sec:prelim-graphs-hypergraphs}
Graphs are combinatorial objects consisting of a set of vertices which 
are connected by edges.
In the following, we begin with the definition of simple and directed graphs before we come 
to multigraphs and directed hypergraphs as in~\cite{gallo93}.

\begin{definition}\label{def:simple-graph}
  A \textit{simple graph} $\Gamma := (V, E)$ is given by a finite set of vertices $V$ and 
  a set of edges $E \subseteq \PP(V)$ with $\abs{e} = 2$ for all $e \in E$.
\end{definition}

\begin{definition}\label{def:directed-graph}
  A \textit{directed graph} $\Gamma := (V, E)$ is given by a finite set of vertices $V$ and 
  a set of edges $E \subseteq V \times V$.
\end{definition}

An edge $\{v, w\}$ in a simple graph can be visualized by a line from $v$ to $w$, 
whereas an edge $(v, w)$ in a directed graph can be visualized by an 
arrow from $v$ to $w$. Further, directed graphs can have self-loops $(v, v)$, which are excluded 
in our definition of simple graphs.

Although edges are modeled differently in both definitions, we 
can describe them uniformly by an adjacency matrix.

\begin{definition}\label{def:adj-matrix}
Let $\Gamma:= (V, E)$ be a simple or directed graph. 
Then two vertices $v, w \in V$ are \textit{adjacent} and we write 
$v \sim w$ if $\{v, w\} \in E$ or $(v, w) \in E$ respectively. 
Further, we define the \textit{adjaceny matrix} $\Adj \in \C^{V \times V}$ by
\[
  {(\Adj)}_{vw} := \begin{cases}
    1 & \text{if $v \sim w$,} \\
    0 & \text{otherwise,}
  \end{cases}
  \qquad 
  \forall v, w \in V.
\]
\end{definition}

By allowing multiple edges between each pair of vertices, 
we can generalize directed graphs to multigraphs.

\begin{definition}\label{def:multi-graph}
A (directed) \textit{multigraph} $\Gamma := (V, E)$ is given by a finite
set of vertices $V$, a finite set of edges $E$ and
two maps
\[
  s \colon E \to V, \qquad r \colon E \to V.
\] 
\end{definition}

In a multigraph, an edge $e \in E$ can be depicted by a directed arrow from 
the source vertex $s(e)$ to the range vertex $r(e)$ as in the case of directed graphs.
Further, in the setting of multigraphs, we will be interested in vertices with only incoming or outgoing edges.

\begin{definition}
  Let $\Gamma := (V, E)$ be a multigraph and $v \in V$. Then
  \begin{enumerate}
    \item $v$ is a \textit{source}, if $v \neq r(e)$ for all $e \in E$.
    \item $v$ is a \textit{sink}, if $v \neq s(e)$ for all $e \in E$.
    \item $v$ is \textit{isolated}, if $v$ is a source and a sink.
  \end{enumerate} 
\end{definition}

By replacing the vertices $s(e)$
and $r(e)$ in a multigraph with arbitrary subsets of vertices, we finally arrive at the the
definition of a hypergraph.

\begin{definition}\label{def:hypergraph}
A (directed) \textit{hypergraph} $\Gamma := (V, E)$ is given by a finite
set of vertices $V$, a finite set of edges $E$ and
source and range maps 
\[
  s \colon E \to \PP(V), \qquad r \colon E \to \PP(V).
\] 
\end{definition}

In a hypergraph, an edge $e \in E$ can be depicted by an arrow from the 
set of source vertices $s(e)$ to the set of range vertices $r(e)$.
Thus, classical directed edges correspond to hyperedge with $\abs{s(e)} = \abs{r(e)} = 1$,
see also \Cref{def:dir-graph-to-hypergraph} and \Cref{def:multigraph-to-hypergraph} below.
Note that we allow empty sets in both the source and range map.

Similar to the adjacency matrix of classical graphs, it is also possible to describe the 
edge structure of a hypergraph using so-called incidence matrices.

\begin{definition}
Let $\Gamma:= (V, E)$ be a hypergraph. 
Then its \textit{incidence matrices} $A_s, A_r \in \C^{V \times E}$ are given by
\[
  {(A_s)}_{ve} = \begin{cases}
    1 & \text{if $v \in s(e)$,} \\
    0 & \text{otherwise,}
  \end{cases}
  \qquad
  {(A_r)}_{ve} = \begin{cases}
    1 & \text{if $v \in r(e)$,} \\
    0 & \text{otherwise,}
  \end{cases}
\]
for all $v \in V$ and $e \in E$.
\end{definition}

It is also possible to generalize the notation of source and sink from multigraphs
to hypergraphs.

\begin{definition}
  Let $\Gamma := (V, E)$ be a hypergraph and $v \in V$. Then
  \begin{enumerate}
    \item $v$ is a \textit{source}, if $v \notin r(e)$ for all $e \in E$.
    \item $v$ is a \textit{sink}, if $v \notin s(e)$ for all $e \in E$.
    \item $v$ is \textit{isolated}, if $v$ is a source and a sink.
  \end{enumerate} 
\end{definition}

Further, the following properties will be used to describe 
hypergraphs with special structure.
 
\begin{definition} Let $\Gamma := (V, E)$ be a hypergraph. Then
\begin{enumerate}
\item $\Gamma$ has \textit{no multi-edges}, if $s(e_1) = s(e_2)$ and $r(e_1) = r(e_2)$ implies $e_1 = e_2$
  for all $e_1, e_2 \in E$.
\item $\Gamma$ is \textit{undirected}, if $s(e) = r(e)$ for all $e \in E$.
\item $\Gamma$ is \textit{$k$-uniform}, if $|s(e)| = |r(e)| = k$ for all $e \in E$.
\end{enumerate}
\end{definition}

Given a simple graph, a directed graph 
or a multi-graph, it is possible to regard it as a hypergraph in a natural way. 
In the following, we define the corresponding source and range maps and show how the resulting hypergraphs 
can be characterized using the properties defined previously.

\begin{definition}\label{def:simple-graph-to-hypergraph}
Let $\Gamma:= (V, E)$ be a simple graph. 
Then we can regard $\Gamma$ as a hypergraph with the 
source and range maps
\[
  s(\{v, w\}) := \{ v, w \},
  \qquad 
  r(\{v, w\}) := \{ v, w \}
  \qquad 
  \forall \{v, w\} \in E.
\]
Conversely, $2$-uniform undirected hypergraphs without multi-edges correspond exactly to simple graphs in this way.
\end{definition} 

\begin{definition}\label{def:dir-graph-to-hypergraph}
Let $\Gamma:= (V, E)$ be a directed graph. 
Then we can regard $\Gamma$ as a hypergraph by defining the source and range maps
\[
  s(v, w) := \{ v \}, 
  \qquad 
  r(v, w) := \{ w \}
  \qquad 
  \forall (v, w) \in E.
\]
Conversely, $1$-uniform hypergraphs without multi-edges correspond exactly to directed graphs in this way.
\end{definition} 

\begin{definition}\label{def:multigraph-to-hypergraph}
Let $\Gamma := (V, E)$ be a multigraph with source
and range maps
\[
  s' \colon E \to V, 
  \qquad 
  r' \colon E \to V.
\]
Then we can regard $\Gamma$ as a hypergraph with the 
new source and range maps
\[
  s(e) := \{s'(e)\},
  \qquad
  r(e)  := \{r'(e)\}
  \qquad 
  \forall e \in E.
\]
Conversely, $1$-uniform hypergraphs correspond exactly to multigraphs in this way.
\end{definition} 

Given any hypergraph, it is always possible to obtain a new hypergraph
by interchanging the source and range map or by interchanging the vertices and edges.
  
\begin{definition}\label{def:opposite-hypergraph}
  Let $\Gamma = (V, E)$ be a hypergraph. Then its \textit{opposite hypergraph}
  is given by $\Gamma^{\op} := (V, E)$ with 
  \[
    s^{\op}(e) := r(e), \quad r^{\op}(e) := s(e) \quad \forall e \in E.
  \]
\end{definition}

\begin{definition}\label{def:dual-hypergraph}
Let $\Gamma = (V, E)$ be a hypergraph. Then its \textit{dual hypergraph} is given by 
$\Gamma^* := (E, V)$ with source and range maps $s^*$ and $r^*$ defined by
\[
  s^*(v) := \{ e \in E \ | \ v \in s(e) \}, \quad r^*(v) := \{ e \in E \ | \ v \in r(e) \}.
\]
\end{definition}

Note that we can use the dual source and range maps $s^*$ and $r^*$ to rewrite 
the source and range maps $s$ and $r$ as
\[
  s(e) = \{ v \in V \ | \ e \in s^*(v) \}, \quad
  r(e) = \{ v \in V \ | \ e \in r^*(v) \} \quad \forall e \in E.
\]
In particular, we have ${(\Gamma^*)}^* = \Gamma$.

\subsection{Compact matrix quantum groups}
Compact quantum groups were first introduced by Woronowicz in~\cite{woronowicz87, woronowicz91}
and are a generalization of classical compact groups to describe
symmetries in the setting of $C^*$-algebras.

In the following, we will focus only on compact matrix quantum groups,
which form a subclass of compact quantum groups analogous to classical matrix groups.
We begin with some basic definitions before we 
come to actions on $C^*$-algebras and the quantum permutation group. 
For more information on compact quantum groups, we refer to~\cite{neshveyev13} and~\cite{timmermann08}.

\begin{definition}
Let $I$ be a finite index set. 
A \textit{compact matrix quantum group} $G := (\A, u)$ consists of a
unital $C^*$-algebra $\A$ and a matrix $u := {(u_{ij})}_{i,j \in I} \in \A \otimes \C^{I \times I}$, such that 
\begin{enumerate}
\item $\A$ is generated by the matrix entries $u_{ij}$ for all $i, j \in I$,
\item $u$ is unitary and $\overline{u} := {(u_{ij}^*)}_{i,j\in I}$ is invertible,
\item there exists a unital $*$-homomorphism $\Delta \colon \A \to \A \otimes \A$ with 
\[
    \Delta(u_{ij}) = \sum_{k \in I} u_{ik} \otimes u_{kj} \quad \forall  i, j \in I.
\]
\end{enumerate}
If $G$ is a compact matrix quantum group, then the $C^*$-algebra $\A$ will be denoted by $C(G)$
and the $*$-algebra generated by the entries $u_{ij}$ will be denoted by $\OO(G)$.
Further, the matrix $u$ is called the fundamental representation of $G$.
\end{definition}

As for classical groups, it is possible to define
subgroups, quotients and isomorphisms for compact matrix quantum groups.
However, to formulate these, it is convenient to first 
introduce morphisms of compact quantum groups.

\begin{definition}
Let $G$ and $H$ be two compact matrix quantum groups.
A \textit{morphism of compact quantum groups} is a unital $*$-homomorphism 
\[
  \varphi \colon C(H) \to C(G),
\]
such that 
\[
  (\varphi \otimes \varphi) \circ \Delta_H = \Delta_G \circ \varphi.
\]
\end{definition}

Note that there exists also a stronger notion of morphism between
compact matrix quantum groups which respects their fundamental representations. However,
we are interested in morphisms of compact quantum groups since we want to be able to compare 
compact matrix quantum groups of different sizes.

\begin{definition}
Let $G$ and $H$ be two compact matrix quantum groups.
\begin{enumerate}
\item $H$ is a \textit{subgroup} of $G$ and we write $H \subseteq G$ if there exists a surjective 
morphism of compact quantum groups $\varphi \colon C(G) \to C(H)$.
\item $H$ is a \textit{quotient} of $G$ if there exists an injective morphism of compact quantum groups
$\varphi \colon C(H) \to C(G)$. In this case, we identify $C(H)$ with a subalgebra of $C(G)$.
\item $G$ and $H$ are \textit{isomorphic} and we write $G = H$ if there exists a bijective
morphism of compact quantum groups $\varphi \colon C(G) \to C(H)$.
Note that in this case, $\varphi^{-1}$ is again a morphism of compact quantum groups.
\end{enumerate} 
\end{definition}

While the comultiplication $\Delta$ of a quantum group generalizes the multiplication of 
a classical group, there exists further structure on 
the $*$-algebra $\OO(G)$ which corresponds to the neutral element
and the inverse map in the classical case.

\begin{remark}\label{rem:counit-antipode}
Let $G$ be a compact matrix quantum group with fundamental representation $u$.
Then there exists a unital $*$-homomorphism $\varepsilon \colon \OO(G) \to \C$ 
called \textit{counit} and a unital anti-homomorphism $S \colon \OO(G) \to \OO(G)$
called \textit{antipode}, which are given by
\[
  \varepsilon(u_{ij}) = \delta_{ij},
  \qquad 
  S(u_{ij}) = u_{ji}^*
  \qquad 
  \forall i, j \in I.
\]
These turn $\OO(G)$ into a Hopf $*$-algebra. For more details, see~\cite{timmermann08}.
\end{remark}

Using the Hopf $*$-algebra structure on $\OO(G)$, we can now define 
actions of compact matrix quantum groups as in~\cite{wang98} and~\cite{bichon03}.

\begin{definition}\label{def:qgroup-action}
Let $G$ be a compact matrix quantum group and $\A$ a unital $C^*$-algebra.
An \textit{action} of $G$ on $\A$ is a unital $*$-homomorphism
$\alpha \colon \A \to \A \otimes C(G)$, such that 
\begin{enumerate}
  \item $(\alpha \otimes \id) \circ \alpha = (\id \otimes \Delta) \circ \alpha$,
  \item there exists a dense $*$-subalgebra $\B \subseteq \A$ with $\alpha(\B) \subseteq \B \otimes \OO(G)$,
  \item $(\id \otimes \varepsilon) \circ \alpha|_\B = \id$.
\end{enumerate}
\end{definition}

Alternatively, the second and third conditions can be replaced by the more 
analytic condition that $(1 \otimes C(G)) \alpha(\A)$ is linearly dense in 
$\A \otimes C(G)$, see for example~\cite{schmidt18}. 
However, \Cref{def:qgroup-action} will be easier to check in our case.

\begin{definition}\label{def:qgroup-faithful}
  Let $\alpha$ be an action of a compact matrix quantum groups $G$ 
  on a $C^*$-algebra $\A$. Then $\alpha$ is \textit{faithful} if 
  for any quotient $H$ of $G$, such that $\alpha|_{C(H)}$ is an action on $\A$, we have
  $C(H) = C(G)$.
\end{definition}
  
Next, we define magic unitaries and the quantum permutation group $S_n^+$, which was first introduced
by Wang~\cite{wang98} as the quantum automorphism group of the finite set $X := \{1, \ldots, n\}$. 
However, we will consider the case of arbitrary finite sets $X$ and denote the 
corresponding quantum permutation group by $S_X^+$. 

\begin{definition}
Let $X$ be a finite set and $\A$ a unital $C^*$-algebra.
An element $u:= (u_{ij}) \in \A \otimes \C^{X \times X}$ is called \textit{magic unitary}, if
\[
  u_{ij}^2 = u_{ij}^* = u_{ij} \quad \forall i, j \in X,
  \qquad
  \sum_{j \in I} u_{ij} = \sum_{j \in I} u_{ji} = 1 \quad \forall i \in X.
\]
\end{definition}

Note that magic unitaries with entries in $\C$ are exactly classical permutation matrices, which 
gives magic unitaries also the name quantum permutation matrices.
These matrices have applications in quantum information via non-local games~\cite{lupini20} 
and concrete magic unitaries have for example been recently studied in~\cite{faross23, nechita23}.
For more information and open problems related to magic unitaries, we refer to~\cite{weber23}.

The quantum permutation group $S_X^+$ is then the compact matrix quantum group with 
a universal magic unitary as its fundamental representation.

\begin{definition}
Let $X$ be a finite set and $\A$ be the universal unital $C^*$-algebra with 
generators $u_{ij}$ for all $i, j \in X$ such that $u := {(u_{ij})}_{i,j\in X}$ is a magic unitary.
Then $S_X^+ := (\A, u)$ is the \textit{quantum permutation group} on $X$.
\end{definition}

In~\cite{wang98}, Wang showed that the quantum permutation group $S_X^+$ is the largest compact matrix quantum 
group which acts on $\C^X$ in the following sense. Thus, we can think of $S_X^+$ as the 
quantum automorphism group of a finite set $X$.

\begin{proposition}[\cite{wang98}]
Let $X$ be a finite set. Then $S_X^+$ acts faithfully on $\C^X$
via $\alpha \colon \C^X \to \C^X \otimes C(S_X^+)$ given by
\[
  \alpha(e_j) = \sum_{j \in X} e_{i} \otimes u_{ij} \quad \forall i \in X,
\]
where $u$ denotes the fundamental representation of $S_X^+$ and ${\{ e_i \}}_{i \in X}$ 
the standard basis of $\C^X$. 
Further, if $G$ is any compact
matrix quantum group acting faithfully on $\C^X$ via the action $\alpha$, then $G \subseteq S_X^+$.
\end{proposition}

When defining quantum automorphism groups of graphs or hypergraphs, we 
will consider intertwiner relations of the form $Au = vA$, where 
$u$ and $v$ are unitary matrices over a $C^*$-algebra and $A$ is a scalar matrix.
The following proposition gives a useful reformulation of these relations.

\begin{proposition}\label{prop:star-intertwiner}
  Let $I, J$ be finite index sets and $\A$ a unital $C^*$-algebra. Further, let $u \in \A \otimes \C^{I \times I}$, $v \in \A \otimes \C^{J \times J}$
  be unitaries and $T \in \C^{J \times I}$.
  Then $Tu = vT$ if and only if $T^* v = u T^*$.
\end{proposition}
\begin{proof}
Let $Tu = vT$. Since $T$ has scalar entries, we have 
\[
  u^*T^* = {(Tu)}^* = {(vT)}^* = T^* v^*.
\] 
Multiplying with $u$ from the left and with $v$ from the right 
yields $T^* v = u T^*$. The other direction follows similarly by replacing $T$ with $T^*$.
\end{proof}
  
Further, we will use the free product of compact matrix quantum groups, which 
was introduced in~\cite{wang95}.

\begin{definition}
Let $G := (\A, u)$ and $H := (\B, v)$ be two compact matrix quantum groups. Then 
the \textit{free product}  $G * H$ is the compact matrix quantum group 
\[
  G * H := (\A * \B, u \oplus v),
\]
where $\A * \B$ denotes the unital free product of $C^*$-algebras and 
we identify $u$ and $v$ with $\A * \B$-valued matrices via the canonical 
inclusions $\A, \B \hookrightarrow \A * \B$.
\end{definition}

\begin{example}
Consider two finite sets $X$ and $Y$ with corresponding
quantum permutations groups $S_X^+$ and $S_Y^+$. Then their free product $S_X^+ * S_Y^+$ is given by 
the universal unital $C^*$-algebra
\[
  C(S_X^+ * S_Y^+) = C^*( u_{i j}, v_{k \ell} \ | \ \text{$(u_{i j})$ and $(v_{k \ell})$ are magic unitaries}),
\]
where the elements $u_{i j}$ are indexed over $X$ and the elements $v_{k \ell}$ are indexed over $Y$. 
Further, the comultiplication is defined by 
\begin{align*}
  \Delta(u_{i j}) &= \sum_{k \in X} u_{ik} \otimes u_{kj}
  \quad
  \forall i, j \in X, \\ 
  \Delta(v_{i j}) &= \sum_{k \in Y} v_{ik} \otimes v_{kj}
  \quad
  \forall i, j \in Y.
\end{align*}
\end{example}

\subsection{Quantum automorphism groups of graphs}
In the case of graphs, there 
exist two different versions of quantum automorphism groups, which 
have been introduced by Bichon~\cite{bichon03} and Banica~\cite{banica05}. 
These quantum groups have for example been further studied in~\cite{chassaniol19, schmidt20, levandovskyy22, bruyn23}
and have been recently generalized to multigraphs in~\cite{goswami23}.
In the following, we begin with the definition of the quantum automorphism groups
of Bichon, before we come to the version of Banica and the recent version for multigraphs
by Goswami-Hossain.
For more information about quantum automorphism groups of graphs, 
we refer to~\cite{schmidt20}.

\begin{definition}\label{def:bichon-qaut}
Let $\Gamma := (V, E)$ be a simple or directed graph and denote with $\A$ the universal
unital $C^*$-algebra with generators $u_{vw}$ for all $v, w \in V$ and 
the following relations:
\begin{enumerate}
\item\label{item:bichon-qaut-1} $u := {(u_{vw})}_{v,w\in V}$ is a magic unitary,
\item\label{item:bichon-qaut-2} $\Adj u = u \Adj$, where $\Adj$ is the adjaceny matrix of $\Gamma$,
\item\label{item:bichon-qaut-3}
$\displaystyle u_{ik} u_{j\ell} = u_{j\ell} u_{ik}$
for all $i \sim j$ and $k \sim \ell$.
\end{enumerate}
Then  $\AutBichon(\Gamma) := (\A, u)$ is \textit{Bichon's quantum automorphism group} of $\Gamma$.
\end{definition}

Note that the second relation $\Adj u = u \Adj$ was originally formulated using a different set of relations. 
However, the following proposition shows that both versions are equivalent.
 
\begin{proposition}\label{prop:bic-intertwiner-equiv}
Let $\Gamma := (V, E)$ be a simple or directed graph, $\A$ a unital
$C^*$-algebra and $u \in \A \otimes \C^{V \times V}$ a magic unitary. Then 
$\Adj u = u \Adj$ is equivalent to 
\begin{align*}
  & u_{ik} u_{j\ell} = u_{j\ell} u_{ik} = 0
  \quad \forall i \sim j, \, k \nsim \ell, \\
  & u_{ik} u_{j\ell} = u_{j\ell} u_{ik} = 0
  \quad \forall i \nsim j, \, k \sim \ell.
\end{align*}
\end{proposition}
\begin{proof}
See~\cite[Proposition 2.1.3.]{schmidt20}.
\end{proof}

Using this reformulation of the intertwiner relations $\Adj u = u \Adj$, 
we can now also prove the following proposition, which will be useful later.

\begin{proposition}\label{prop:bic-sum-one}
Let $\Gamma:= (V, E)$ be a simple or directed graph and denote
with $u$ the fundamental representation of $\AutBichon(\Gamma)$.
Then
  \[
    \sum_{\substack{k, \ell \in V \\ k \sim \ell}} u_{i k} u_{j \ell} = \sum_{\substack{k, \ell \in V \\ k \sim \ell}} u_{k i} u_{\ell j} = 1 \quad \forall i \sim j.
  \]
\end{proposition}
\begin{proof}
Using the \Cref{prop:bic-intertwiner-equiv} and the fact that $u$ is a magic 
unitary, we obtain
\[
  \sum_{\substack{k, \ell \in V \\ k \sim \ell}} u_{i k} u_{j \ell}
  =
  \sum_{k, \ell \in V} u_{i k} u_{j \ell}
  =
  \bigg( \sum_{k \in V} u_{i k} \bigg) 
  \bigg( \sum_{\ell \in V} u_{j \ell} \bigg)
  = 1 \cdot 1 = 1
\]
for all $i, j \in V$ with $i \sim j$,
see also~\cite{schmidt18}. The other equality follows similarly.
\end{proof}

By dropping Relation~\ref{item:bichon-qaut-3} in \Cref{def:bichon-qaut}, we obtain 
the quantum automorphism group of Banica, which is often studied instead of Bichon's version.

\begin{definition}
Let $\Gamma := (V, E)$ be a simple or directed graph and denote with $\A$ the universal
unital $C^*$-algebra with generators $u_{vw}$ for all $v, w \in V$ and 
the following relations:
\begin{enumerate}
\item $u := {(u_{vw})}_{v,w\in V}$ is a magic unitary,
\item $\Adj u = u \Adj$, where $\Adj$ is the adjaceny matrix of $\Gamma$,
\end{enumerate}
Then  $\AutBanica(\Gamma) := (\A, u)$ is \textit{Banica's quantum automorphism group} of $\Gamma$.
\end{definition}

However, we will mainly be interested in Bichon's version 
and its following generalization to multigraphs as 
recently defined by Goswami-Hossain~\cite{goswami23}.

\begin{definition}\label{def:qaut-GH}
Let $G:= (V, E)$ be a multigraph without isolated vertices.
Denote with $\A$ the universal unital $C^*$-algebra 
with generators $u_{ef}$ for all $e, f \in E$
and the following relations.
\begin{enumerate}
\item\label{def:qaut-GH-1} The matrix $u := {(u_{ef})}_{e,f \in E}$ is a magic unitary.
\item\label{def:qaut-GH-2} Let $v \in V$ and $e_1, e_2 \in E$. Then 
\begin{align*}
  \sum_{\substack{f \in E \\ s(f) = v}} u_{e_1 f}
  =
  \sum_{\substack{f \in E \\ s(f) = v}} u_{e_2 f}
  \quad 
  &\text{if $s(e_1) = s(e_2)$,} \\
  \sum_{\substack{f \in E \\ r(f) = v}} u_{e_1 f}
  =
  \sum_{\substack{f \in E \\ r(f) = v}} u_{e_2 f}
  \quad 
  &\text{if $r(e_1) = r(e_2)$.}
\end{align*}
\item\label{def:qaut-GH-3} Let $e, f \in E$. Then $u_{ef} = 0$ if
\begin{itemize}
  \item $s(e)$ is neither a source nor a sink and $s(f)$ is a source,
  \item $r(e)$ is neither a source nor a sink and $r(f)$ is a sink.
\end{itemize}
\item\label{def:qaut-GH-4} Let $v \in V$ be neither a source nor a sink and $e_1, e_2 \in E$ such that 
$s(e_1) = r(e_2)$ is neither a source nor a sink. Then
\[
  \sum_{\substack{f \in E \\ s(f) = v}} u_{e_1 f} = \sum_{\substack{f \in E \\ r(f) = v}} u_{e_2 f}. 
\]
\end{enumerate}
Then $\AutGHBichon(\Gamma) := (\A, u)$ is the \textit{quantum automorphism group of $\Gamma$ by
Goswami-Hossa in Bichon's sense}.
\end{definition}
  
Note that we added the magic unitary relation $u_{ef}^2 = u_{ef}$ and swapped the conditions in Relation~\ref{def:qaut-GH-3}
in contrast to the original definition in~\cite[Definition 4.26]{goswami23}.
  
\subsection{Graph and hypergraph \texorpdfstring{$C^*$}{C*}-algebras}

Graph $C^*$-algebras are a well-studied class of $C^*$-algebras which are 
defined in terms of an underlying graph and generalize Cuntz-Krieger algebras~\cite{cuntz80}. 
They include many concrete examples like
matrix algebras, continuous functions on the circle or Cuntz algebras.
In the following, we define graph $C^*$-algebras only for finite directed graphs. See~\cite{raeburn05} for more information
and the general case of infinite graphs.

\begin{definition}
Let $\Gamma:= (V, E)$ be a directed graph. The \textit{graph $C^*$-algebra}
$C^*(\Gamma)$ is the universal $C^*$-algebra generated
by mutually orthogonal projections $p_v$ for all $v \in V$ and 
partial isometries $s_e$ with orthogonal ranges for all $e \in E$ such that 
\begin{enumerate}
\item $\displaystyle s_{(v,w)}^* s_{(v,w)} = p_{w}$ for all $(v, w) \in E$,
\item $\displaystyle s_{(v,w)} s_{(v,w)}^* \leq p_{v}$ for all $(v, w) \in E$,
\item $\displaystyle p_v \leq \sum_{\substack{e \in E \\ v = s(e)}} s_e s_e^*$ for all $v \in V$ which are not a sink.
\end{enumerate}
\end{definition}

In~\cite{schmidt18}, Schmidt and Weber then
showed that Banica's quantum automorphism group 
$\AutBanica(\Gamma)$ of a finite graph $\Gamma$
acts maximally on the corresponding graph $C^*$-algebra $C^*(\Gamma)$.

\begin{proposition}[\cite{schmidt18}]
Let $\Gamma:= (V, E)$ be a directed graph. 
Then $\AutBanica(\Gamma)$ acts maximally on $C^*(\Gamma)$ via 
$\alpha \colon C^*(\Gamma) \to C^*(\Gamma) \otimes C(\AutBanica(\Gamma))$ given by
\begin{alignat*}{3}
  \alpha(p_v) &= \sum_{w \in V} u_{vw} \otimes p_{w} 
  &\qquad& \forall v \in V, \\
  \alpha(s_{(v_1, v_2)}) &= \sum_{(v_2, w_2) \in E} u_{v_1 v_2} u_{w_1 w_2} \otimes s_{(v_2, w_2)} 
  & & \forall (v_1, w_1) \in E.
\end{alignat*}
\end{proposition}

See also~\cite{joardar18} for more on actions
of quantum groups on graph $C^*$-algebra.

Recently, Trieb-Weber-Zenner~\cite{trieb24} introduced hypergraph $C^*$-algebras 
by replacing the graph definition of graph $C^*$-algebras with a hypergraph. 

\begin{definition}\label{def:hypergraph-Cstar-alg}
Let $\Gamma:= (V, E)$ be a hypergraph. Then the \textit{hypergraph $C^*$-algebra} 
$C^*(\Gamma)$ is the universal $C^*$-algebra generated
by mutually orthogonal projections $p_v$ for all $v \in V$ and 
partial isometries $s_e$ for all $e \in E$ such that 
\begin{enumerate}
\item\label{def:hypergraph-Cstar-alg-R1} $\displaystyle s_e^* s_f = \delta_{ef} \sum_{\substack{v \in V \\ v \in r(e)}} p_v$ for all $e, f \in E$,
\item\label{def:hypergraph-Cstar-alg-R2} $\displaystyle s_e s_e^* \leq \sum_{\substack{v \in V \\ v \in s(e)}} p_v$ for all $e \in E$,
\item\label{def:hypergraph-Cstar-alg-R3} $\displaystyle p_v \leq \sum_{\substack{e \in E \\ v \in s(e)}} s_e s_e^*$ for all $v \in V$ which are not a sink.
\end{enumerate}
\end{definition}

If a directed graph is encoded as a hypergraph as in \Cref{def:dir-graph-to-hypergraph}, then 
the corresponding hypergraph $C^*$-algebra agrees with the classical graph $C^*$-algebra.
However, hypergraph $C^*$-algebras include also new examples of 
non-nuclear $C^*$-algebras.
For more on the nuclearity of hypergraph $C^*$-algebras, we refer to the recent work by Schäfer-Weber~\cite{schaefer24}, where
the nuclearity of hypergraph $C^*$-algebras is characterized in terms of minors of the underlying hypergraph.

\section{Quantum symmetries of hypergraphs}\label{sec:qsym-hypergraphs}

\subsection{Automorphism groups of hypergraphs}\label{sec:classic-hypergraph-aut}
We first recall the definition of 
the classical automorphism group of a hypergraph and characterize it in terms of permutation matrices, 
before we introduce the quantum automorphism group of a hypergraph.

\begin{definition}\label{def:classical-aut}
Let $\Gamma:= (V, E)$ be a hypergraph. Then its \textit{automorphism group} is given by
\[
  \Aut(\Gamma) := \left\{ (\sigma, \tau) \in S_V \times S_E \ \middle| \ \forall e \in E\colon \text{$\sigma s(e) = s(\tau e)$ and $\sigma r(e) = s(\tau e)$}\right\},
\]
where $\sigma \in S_V$ acts on $\PP(V)$ by 
$
  \sigma(\{ v_1, \ldots v_k \}) := \{ \sigma v_1, \ldots, \sigma v_k \}.
$
\end{definition}

Thus, a hypergraph automorphism consists of a permutation of the vertices and a permutation of the edges, 
which are compatible by respecting the source and range maps.
Next, we show how these compatibility conditions can be  
reformulated when both permutations are given as permutation matrices.

\begin{definition}\label{def:perm-repr}
Let $X$ be a finite set. The \textit{permutation representation} 
of $S_X$ is given by $S_X \to \C^{X \times X}$, $\sigma \mapsto P_\sigma$ with 
\[
  {(P_\sigma)}_{ij} = \delta_{i\sigma(j)} \quad \forall i, j \in X.
\]
\end{definition}

Note that the permutation representation is faithful, such that we have an 
embedding $S_X \hookrightarrow \C^{X \times X}$ given by permutation matrices.

\begin{proposition}\label{prop:perm-intertwiner}
Let $X$, $Y$ be finite sets, $\sigma \in S_X$, $\tau \in S_Y$ and $A \in \C^{X \times Y}$. Then 
$A P_\tau = P_\sigma A$ if and only if 
$A_{i\tau(j)} = A_{\sigma^{-1}(i)j}$ for all $i \in X$ and $j \in Y$.
\end{proposition}
\begin{proof}
Using the definition of $P_\sigma$ and $P_\tau$, we compute
\[
  {(A P_\tau)}_{ij} = \sum_{k \in Y} A_{ik} {(P_\tau)}_{kj}
  = \sum_{k \in Y} \delta_{k\tau(j)} A_{ik}
  = A_{i\tau(j)} \quad \forall i \in X, j \in Y, 
\]
\[
  {(P_\sigma A)}_{ij} = \sum_{k \in X} {(P_\sigma)}_{ik} A_{kj}
  = \sum_{k \in X} \delta_{i\sigma(k)} A_{kj}
  = A_{\sigma^{-1}(i)j} \quad \forall i \in X, j \in Y.
\]
Thus, 
\[
  A P_\tau = P_\sigma A 
  \ \Longleftrightarrow \ A_{i\tau(j)} = A_{\sigma^{-1}(i)j} \quad \forall i \in X, j \in Y. 
\]
\end{proof}

By applying the previous proposition to the incidence matrices of a hypergraph, we obtain 
exactly the compatibility conditions in \Cref{def:classical-aut}. 

\begin{proposition}\label{prop:intertwiner-char}
Let $\Gamma := (V, E)$ be a hypergraph and $(\sigma, \tau) \in S_V \times S_E$. Then 
\begin{enumerate}
  \item $\sigma(s(e)) = s(\tau(e))$ for all $e \in E$ if and only if $A_s P_\tau = P_\sigma A_s$,
  \item $\sigma(r(e)) = r(\tau(e))$ for all $e \in E$ if and only if $A_r P_\tau = P_\sigma A_r$,
\end{enumerate}
where $A_s$ and $A_r$ are the incidence matrices of $\Gamma$.
\end{proposition}
\begin{proof}
We prove only the first statement about the source map $s$. The second statement about the range map $r$ follows similarly.
By \Cref{prop:perm-intertwiner}, we have 
\[
  A_s P_\tau = P_\sigma A_s \ \Longleftrightarrow \ {(A_s)}_{v\tau(e)} = {(A_s)}_{\sigma^{-1}(v) e} \quad \forall v \in V, e \in E
\]
and by the definition of $A_s$, the right-hand side is equivalent to 
\[
  v \in s(\tau(e)) \ \Longleftrightarrow \ \sigma^{-1}(v) \in s(e) \quad \forall v \in V, e \in E.
\]
However, $\sigma^{-1}(v) \in s(e)$ can be rewritten to $v \in \sigma s(e)$, such that we obtain
\[
  v \in s(\tau(e)) \ \Longleftrightarrow \ v \in \sigma s(e) \quad \forall v \in V, e \in E.
\]
This states exactly
\[
   s(\tau(e)) = \sigma(s(e)) \quad \forall e \in E.
\]
\end{proof}

\subsection{Quantum automorphism groups of hypergraphs}

Using the characterization of classical hypergraph automorphisms in terms 
of permutations matrices, we can now define the quantum automorphism group of 
a hypergraph.

\begin{definition}\label{def:hypergraph-qaut}
Let $\Gamma := (V, E)$ be a hypergraph and denote with $\A$ the universal unital $C^*$-algebra 
with generators $u_{vw}$ for all $v, w \in V$ and $u_{ef}$ for all $e,f \in E$
such that
\begin{enumerate}
  \item $u_V := {(u_{v w})}_{v, w \in V}$ and $u_E := {(u_{e f})}_{e, f \in E}$ are magic unitaries,
  \item $A_s u_E = u_V A_s$ and $A_r u_E = u_V A_r$, where $A_s, A_r \in \C^{V \times E}$ are the incidence matrices of $\Gamma$.
\end{enumerate}
Then $\Aut^+(\Gamma) := (\A, u_V \oplus u_E)$ is the \textit{quantum automorphism group of $\Gamma$}.
\end{definition}

Intuitively, we replace the permutations $\sigma \in S_V$ and $\tau \in S_E$ 
in the definition of the classical automorphism group by quantum permutation matrices $u_V$ and $u_E$. 
The compatibility conditions between these two matrices can then be expressed by the intertwining relations
from the previous sections.

Before we show that the previous definition
of $\Aut^+(\Gamma)$ indeed generalizes the classical automorphism group $\Aut(\Gamma)$, 
we first comment on the relations in \Cref{def:hypergraph-qaut}.

\begin{remark}\label{rem:hypergraph-qaut-rels}
Recall that the magic unitary relations of $u_V$ are given by 
\[
  u_{vw}^2 = u_{vw}^* = u_{vw} \quad \forall v, w \in V,
  \qquad
  \sum_{w \in V} u_{vw} = \sum_{w \in V} u_{wv} = 1
  \quad \forall v \in V.
\]
Similarly, the magic relations of $u_E$ are given by 
\[
  u_{ef}^2 = u_{ef}^* = u_{ef} \quad \forall e, f \in E,
  \qquad
  \sum_{f \in E} u_{ef} = \sum_{f \in e} u_{fe} = 1
  \quad \forall e \in E.
\]
and the intertwiner relations $A_s u_E = u_V A_s$ and $A_r u_E = u_V A_r$
can be written as
\[
  \sum_{\substack{f \in E \\ v \in s(f)}} u_{fe} 
  = 
  \sum_{\substack{w \in V \\ w \in s(e)}} u_{vw},  
  \qquad
  \sum_{\substack{f \in E \\ v \in r(f)}} u_{fe} 
  = 
  \sum_{\substack{w \in V \\ w \in r(e)}} u_{vw}
  \qquad
  \forall v \in V, e \in E.
\]
Note that $A_s^*$ and $A_r^*$ are also intertwiners by \Cref{prop:star-intertwiner} since $u_V$ and $u_E$ are unitaries. 
Thus, we have the additional relations $A_s^* u_V = u_E A_s^*$ and $A_r^* u_V = u_E A_r^*$,
which can be written as 
\[
  \sum_{\substack{w \in V \\ w \in s(e)}} u_{wv} 
  = 
  \sum_{\substack{f \in E \\ v \in s(f)}} u_{ef}, 
  \qquad
  \sum_{\substack{w \in V \\ w \in r(e)}} u_{wv}
  =
  \sum_{\substack{f \in E \\ v \in r(f)}} u_{ef} 
  \qquad
  \forall v \in V, e \in E.
\]
\end{remark}

\begin{remark}
Denote with $\A$ the $C^*$-algebra $C(\Aut^+(\Gamma))$.
Then the magic unitary relations and intertwiner relations are compatible with the comultiplication in the sense that there 
exists a well-defined unital $*$-homomorphism $\Delta \colon \A \to \A \otimes \A$ with
\[
  \Delta(u_{v_1 v_2}) = \sum_{w \in V} u_{v_1 w} \otimes u_{w \otimes v_2} \quad \forall v_1, v_2 \in V,
\]
\[
  \Delta(u_{e_1 e_2}) = \sum_{f \in E} u_{e_1 f} \otimes u_{f \otimes e_2} \quad \forall e_1, e_2 \in E.
\]
Thus, $\Aut^+(\Gamma)$ is well-defined as a compact matrix quantum group. For convenience, we will show the existence of $\Delta \colon \A \to \A \otimes \A$ explicitly using the universal property of 
$\A$. Thus, we have to verify that the elements $\Delta(u_{v_1 v_2})$ and $\Delta(u_{e_1 e_2})$ satisfy
the magic unitary and intertwiners relations from \Cref{def:hypergraph-qaut}.
For the magic unitary relations, we compute
\begin{align*}
  {\Delta(u_{v_1 v_2})}^2 
  &=
  \sum_{w_1, w_2 \in V} u_{v_1 w_1} u_{v_1 w_2} \otimes u_{w_1 v_2} u_{w_2 v_2} 
  =
  \sum_{w \in V} u_{v_1 w} \otimes u_{w \otimes v_2}
  =
  \Delta(u_{v_1 v_2}), \\
  {\Delta(u_{v_1 v_2})}^*
  &=
  \sum_{w \in V} u_{v_1 w}^* \otimes u_{w \otimes v_2}^*
  =
  \sum_{w \in V} u_{v_1 w} \otimes u_{w \otimes v_2}
  =
  \Delta(u_{v_1 v_2})
\end{align*}
for all $v_1, v_2 \in V$ and
\begin{align*}
  \sum_{v_2 \in V} \Delta(u_{v_1 v_2}) &= 
  \sum_{v_2, w \in V} u_{v_1 w} \otimes u_{w v_2} \\
  &= \sum_{w \in V} u_{v_1 w} \otimes \left(\sum_{v_2 \in V} u_{w v_2}\right)
  = \left(\sum_{w \in V} u_{v_1 w}\right) \otimes 1
  = 1 \otimes 1
\end{align*}
for all $v_1 \in V$. Similarly, one computes 
\[
  \sum_{v_2 \in V} \Delta(u_{v_2 v_1}) =  \sum_{v_2, w \in V} u_{v_2 w} \otimes u_{w v_1} = 1 \otimes 1,
\]
such that $\Delta(u_V) := {(\Delta(u_{vw}))}_{v,w\in V}$ is a magic unitary.
By replacing the vertex set $V$ with the edge set $E$, the same computation shows that 
$\Delta(u_E) := {(\Delta(u_{ef}))}_{e,f\in E}$ is also a magic unitary.
To check the intertwiner relations, we compute 
\begin{align*}
  \sum_{\substack{f \in E \\ v \in s(f)}} \Delta(u_{fe}) &=
  \sum_{\substack{f \in E \\ v \in s(f)}} \sum_{g \in E} u_{fg} \otimes u_{ge} =
  \sum_{g \in E} \bigg( \sum_{\substack{f \in E\\ v \in s(f)}} u_{fg} \bigg) \otimes u_{ge} =
  \sum_{g \in E} \sum_{\substack{x \in V \\ x \in s(g)}} u_{vx} \otimes u_{ge} \\ &=
  \sum_{x \in V} u_{vx} \otimes \bigg( \sum_{\substack{g \in E \\ x \in s(g)}} u_{ge} \bigg) =
  \sum_{x \in V} u_{vx} \otimes \bigg( \sum_{\substack{w \in V \\ w \in s(e)}} u_{xw} \bigg) \\ &= 
  \sum_{\substack{w \in V \\ w \in s(e)}} \sum_{x \in V} u_{vx} \otimes  u_{xw}  =  
  \sum_{\substack{w \in V \\ w \in s(e)}} \Delta(u_{vw})
\end{align*}
for all $v \in V$, $e \in E$.
Hence, $A_s \Delta(u_E) = \Delta(u_V) A_s$. By replacing the source map $s$ with the 
range map $r$, the same computation also shows $A_r \Delta(u_E) = \Delta(u_V) A_r$.
\end{remark}

\begin{remark} Let $u := u_V \oplus u_E$ and define the block matrix 
\[
    A := \begin{pmatrix}
      0     & A_s \\
      A_r^* & 0
    \end{pmatrix} \in \C^{(V \sqcup E) \times (V \sqcup E)}.
\]
Then $Au = uA$ is equivalent to $A_s u_E = u_V A_s$ and $A_r^* u_V = u_E A_r^*$, 
where the second equation is again equivalent to $A_r u_E = u_V A_r$ by \Cref{prop:star-intertwiner}.
Therefore, the relations $A_s u_E = u_V A_s$ and $A_r u_E = u_V A_r$ in \Cref{def:hypergraph-qaut} can be formulated
as the single intertwiner relation $Au = uA$.
\end{remark}

The following proposition then shows that the quantum
automorphism group $\Aut^+(\Gamma)$ generalizes the 
classical automorphism group $\Aut(\Gamma)$ in the sense of compact matrix quantum groups.

\begin{proposition}
Let $\Gamma$ by a hypergraph. Then $\Spec C(\Aut^+(\Gamma)) \cong \Aut(\Gamma)$ as finite groups.
\end{proposition}
\begin{proof}
Let $\Gamma := (V, E)$ and denote with $\A$ the $C^*$-algebra $C(\Aut^+(\Gamma))$. Then $\Spec \A$ is a group 
with multiplication given by
\[
    \varphi * \psi := (\varphi \otimes \psi) \circ \Delta 
    \qquad 
    \forall \varphi, \psi \in \Spec A
\]
and it is isomorphic to a subgroup of unitary matrices $G \subseteq \C^{(V \sqcup E) \times (V \sqcup E)}$ via 
\[
  \varphi \in \Spec \A \ \longleftrightarrow \ \varphi(u) := (\varphi(u_{ij})) \in \C^{(V \sqcup E) \times (V \sqcup E)}.
\]
See for example~\cite[Proposition 6.1.11]{timmermann08}. Further, we have the decomposition
$u = u_E \oplus u_V$,
such that $\varphi(u) = \varphi(u_E) \oplus \varphi(u_V)$ is given by a pair of matrices $\varphi(u_E)$ and $\varphi(u_V)$.
Since $u_V$ and $u_E$ are magic unitaries, $\varphi(u_E)$ and $\varphi(u_V)$ are exactly permutation matrices, 
which correspond to a pair of permutations $(\sigma, \tau) \in S_V \times S_E$ via the 
permutation representation from \Cref{def:perm-repr}. 
\Cref{prop:intertwiner-char} then implies that $(\sigma, \tau)$ are exactly 
automorphisms of $\Gamma$.
\end{proof}

\subsection{Examples of hypergraphs with maximal quantum symmetry}\label{sec:example-max-qsym}
Before we come to examples of quantum automorphism groups of hypergraphs, 
we first show that these quantum groups are always subgroups of $S_V^+ * S_E^+$.

\begin{proposition}\label{prop:free-prod-subgroup}
Let $\Gamma:= (V, E)$ be a hypergraph.
Then 
\[
  \Aut^+(\Gamma) \subseteq S_V^+ * S_E^+.
\]
\end{proposition}
\begin{proof}
Denote with $u$ the fundamental representation of $\Aut^+(\Gamma)$
and with $\widehat{u}$ the fundamental 
representation of $S_V^+ * S_E^+$. 
Since $u_V$ and $u_E$ are magic unitaries, the universal property of $C(S_V^+ * S_E^+)$ 
implies the existence of 
a unital $*$-homomorphism $\varphi \colon C(S_V^+ * S_E^+) \to C(\Aut^+(\Gamma))$ with
\[
  \widehat{u}_{vw} \mapsto u_{vw} 
  \quad \forall v, w \in V,
  \qquad
  \widehat{u}_{ef} \mapsto u_{ef} 
  \quad \forall e, f \in E.
\]
This $*$-homomorphism is surjective since $C(\Aut^+(\Gamma))$ is generated by the entries
of $u_V$ and $u_E$. Further, it is a morphism of compact quantum groups, because 
\[
  \Delta(\varphi(\widehat{u}_{vw})) = \Delta({u}_{vw}) = \sum_{x \in V} {u}_{vx} \otimes {u}_{xw}
  = \sum_{x \in V} \varphi(\widehat{u}_{vx}) \otimes \varphi(\widehat{u}_{xw})
  = (\varphi \otimes \varphi)(\Delta(\widehat{u}_{vw}))
\]
for all $v, w \in V$ and similarly $\Delta(\varphi(\widehat{u}_{ef})) = \Delta(\varphi(\widehat{u}_{ef}))$
for all $e, f \in E$.
Thus, $\Aut^+(\Gamma) \subseteq S_V^+ * S_E^+$.
\end{proof}

Next, we construct a concrete family of hypergraphs $\Gamma_{n,m}$
for which equality in the previous proposition holds.
Thus, these hypergraphs have maximal possible quantum symmetries
in the sense of \Cref{def:hypergraph-qaut}.

\begin{definition}\label{def:max-sym-hgraph}
Let $n, m \in \N$. Define the hypergraph $\Gamma_{n, m} := (V, E)$ with 
vertices $V = \{1, \ldots, n\}$, edges $E = \{1, \ldots, m\}$ and source and range maps
\[
    s(e) := V, \quad r(e) := V \quad \forall e \in E.
\]
\end{definition}

\begin{proposition}
Let $n, m \in \N$ and $\Gamma_{n, m} := (V, E)$ be the hypergraph from \Cref{def:max-sym-hgraph}. Then 
\[
  \Aut^+(\Gamma_{n,m}) = S_V^+ * S_E^+.
\]
\end{proposition}
\begin{proof}
Denote with $u$ the fundamental representation of $\Aut^+(\Gamma_{n,m})$ and 
with $\widehat{u}$ the fundamental representation of $S_V^+ * S_E^+$.
By the proof of \Cref{prop:free-prod-subgroup}, we have $\Aut^+(\Gamma_{n,m}) \subseteq S_V^+ * S_E^+$
via a unital $*$-homomorphism $C(S_V^+ * S_E^+) \to C(\Aut^+(\Gamma_{n,m}))$ with 
\[
  \widehat{u}_{vw} \mapsto u_{vw}
  \quad \forall v, w \in V,
  \qquad
  \widehat{u}_{ef} \mapsto u_{ef}
  \quad \forall e, f \in E.
\]
To show the other inclusion, we construct the inverse $*$-homomorphism
using the universal property of $C(\Aut^+(\Gamma_{n,m}))$. Thus, we have to show that 
$\widehat{u}_V$ and $\widehat{u}_E$ satisfy the relations from \Cref{def:hypergraph-qaut}.
However, $\widehat{u}_V$ and $\widehat{u}_{E}$ are magic unitaries by definition
and we compute
\[
  \sum_{\substack{f \in E \\ v \in s(f)}} u_{fe} 
  =
  \sum_{f \in E} u_{fe}
  = 
  1 
  =
  \sum_{w \in V} u_{vw}
  = 
  \sum_{\substack{w \in V \\ w \in s(e)}} u_{vw}
  \quad 
  \forall v \in V, e \in E,
\]
such that $A_s \widehat{u}_E = \widehat{u}_V A_s$ by \Cref{rem:hypergraph-qaut-rels}. Similarly, one shows
$A_r \widehat{u}_E = \widehat{u}_V A_r$ by replacing the source map $s$ 
with the range map $r$. Thus, the $*$-homomorphism from 
\Cref{prop:free-prod-subgroup} is invertible, such that
\[
  \Aut^+(\Gamma_{n,m}) = S_V^+ * S_E^+.
\]
\end{proof}

\subsection{Opposite and dual hypergraphs}\label{sec:opp-dual-hypergraphs}

Next, we compute the quantum automorphism groups of the opposite and the dual of a hypergraph.
Recall from \Cref{def:dual-hypergraph} that the opposite hypergraph $\Gamma^{\op}$ is obtained 
by interchanging the source map and the range map of a hypergraph $\Gamma$. 
In the classical case, both $\Gamma$ and $\Gamma^{op}$ have the same automorphism group.

\begin{proposition}
Let $\Gamma:= (V, E)$ be a hypergraph. Then $\Aut(\Gamma) = \Aut(\Gamma^{op})$.
\end{proposition}
\begin{proof}
The statement follows directly, since 
\begin{align*}
  \sigma s(e) = s(\tau e) \quad & 
  \Longleftrightarrow \quad  
  \sigma r^{\op}(e) = r^{\op}(\tau e)
  \qquad
  \forall e \in E, \\
  \sigma r(e) = r(\tau e) \quad & 
  \Longleftrightarrow \quad  
  \sigma s^{\op}(e) = s^{\op}(\tau e)
  \qquad
  \forall e \in E
\end{align*}
for any pair of permutations $(\sigma, \tau) \in S_V \times S_E$.
\end{proof}

The previous proposition can be generalized directly into the quantum setting.

\begin{proposition}\label{prop:opposite-isom}
Let $\Gamma$ be a hypergraph. Then $\Aut^+(\Gamma) = \Aut^+(\Gamma^{op})$.
\end{proposition}
\begin{proof}
Let $\Gamma := (V, E)$. Denote with $u$ the fundamental representation of $\Aut^+(\Gamma)$ and with 
$\widehat{u}$ the fundamental representation of $\Aut^+(\Gamma^{op})$. By 
definition, we have $A_{s^{\op}} = A_r$ and $A_{r^{\op}} = A_s$, such that 
the entries of $u$ and $\widehat{u}$ satisfy exactly the same relations.
Hence, there exists a $*$-isomorphism $\varphi \colon C(\Aut^+(\Gamma)) \to C(\Aut^+(\Gamma^{op}))$ 
with 
\[
  u_{vw} \mapsto \widehat{u}_{vw} \quad \forall v,w \in V,
  \qquad
  u_{ef} \mapsto \widehat{u}_{ef} \quad \forall e,f \in E
\]
by the universal properties of $C(\Aut^+(\Gamma))$ and $C(\Aut^+(\Gamma^{op}))$.
Further, it is a morphism of compact quantum groups since 
\begin{align*}
  \Delta(\varphi(u_{vw})) &= \sum_{x \in V} \widehat{u}_{xw} \otimes \widehat{u}_{xw} = (\varphi \otimes \varphi)(\Delta(u_{vw})) \\
  \Delta(\varphi(u_{ef})) &= \sum_{g \in V} \widehat{u}_{eg} \otimes \widehat{u}_{gf} = (\varphi \otimes \varphi)(\Delta(u_{ef}))
\end{align*}
for all $v, w \in V$ and $e, f \in E$. Thus, $\Aut^+(\Gamma) = \Aut^+(\Gamma^{op})$.
\end{proof}

Next, we consider dual hypergraphs. Recall from \Cref{def:dual-hypergraph} that the dual $\Gamma^*$ of a hypergraph $\Gamma$
is obtained by interchanging the vertices and edges.
As in the case of the opposite hypergraph, a hypergraph and its dual have isomorphic
classical automorphism groups.

\begin{proposition}\label{prop:dual-isom}
Let $\Gamma$ be a hypergraph. Then
\[
  \Aut(\Gamma^*) = \left\{ (\tau, \sigma) \ \middle| \ (\sigma, \tau) \in \Aut(\Gamma) \right\}.
\]
In particular, $\Aut(\Gamma) \cong \Aut(\Gamma^*)$.
\end{proposition}
\begin{proof}
Let $\Gamma := (V, E)$ and $(\sigma, \tau) \in \Aut(\Gamma)$. Then  
\[
  \begin{split}
  \tau s^*(v) 
  &= \{ \tau e \ | \ e \in E, v \in s(e) \}
  = \{ e \in E \ | \ v \in s(\tau^{-1} e) \} \\
  &= \{ e \in E \ | \ v \in \sigma^{-1} s(e) \}
  = \{ e \in E \ | \ \sigma v \in s(e) \} \\
  &= s^*(\sigma v).
  \end{split}
\]
Similarly, one shows $\tau r^*(v) = r^*(\sigma v)$ by replacing $s$ with $r$.
Thus, $(\tau, \sigma) \in \Aut(\Gamma^*)$. Conversely, let $(\tau, \sigma) \in \Aut(\Gamma^*)$.
Then 
\[
  \begin{split}
  \sigma s(e) 
  &= \{ \sigma v \ | \ v \in V, e \in s^*(v) \}
  = \{ v \in V \ | \ e \in s^*(\sigma^{-1} v) \} \\
  &= \{ v \in V \ | \ e \in \tau^{-1} s^*(v) \} 
  = \{ v \in V \ | \ \tau e \in  s^*(v) \} \\
  &= s(\tau e). 
  \end{split}
\]
Again, we can interchange $s$ and $r$ to obtain $\sigma r(e) = r(\tau e)$. 
Hence, $(\sigma, \tau) \in \Aut(\Gamma)$. 
The isomorphism $\Aut(\Gamma) \cong \Aut(\Gamma^*)$ is given by $(\sigma, \tau) \mapsto (\tau, \sigma)$.
\end{proof}

It is again possible to generalize the previous proposition to the case of quantum groups.

\begin{proposition}
Let $\Gamma$ be a hypergraph. Then $\Aut^+(\Gamma) = \Aut^+(\Gamma^*)$.
\end{proposition}
\begin{proof}
Let $\Gamma := (V, E)$ and denote with $u$ the fundamental representation
of $\Aut^+(\Gamma)$ and with $\widehat{u}$ the fundamental representation of 
$\Aut^+(\Gamma^*)$. We begin by constructing a $*$-isomorphism
$\varphi \colon C(\Aut^+(\Gamma)) \to C(\Aut^+(\Gamma^*))$ with 
\[
  \varphi(u_{vw}) = \widehat{u}_{uw} \quad \forall v,w\in V, \qquad
  \varphi(u_{ef}) = \widehat{u}_{ef} \quad \forall e,f\in E
\]
using the universal properties of $C(\Aut^+(\Gamma))$ and $C(\Aut^+(\Gamma^*))$.
Hence, we have to show that entries of $u$ and $\widehat{u}$ satisfy the same relations.
Since $u_V$, $u_E$, $\widehat{u}_E$ and $\widehat{u}_V$ are all magic unitaries, 
it only remains to show that 
\[
  A_s u^{(1)} = u^{(2)} A_s \quad \Longleftrightarrow \quad A_{s^*} u^{(2)} = u^{(1)} A_{s^*},
\]
\[
  A_r u^{(1)} = u^{(2)} A_r \quad \Longleftrightarrow \quad A_{r^*} u^{(2)} = u^{(1)} A_{r^*}
\]
for two arbitrary magic unitaries $u^{(1)}$ and $u^{(2)}$ indexed by $V$ and $E$ respectively.
However, this follows directly from \Cref{prop:star-intertwiner}, since
\[
  {(A_{s^*})}_{ev}
  = \begin{cases}
    1 & \text{if $e \in s^*(v)$} \\
    0 & \text{otherwise}
  \end{cases}
  = \begin{cases}
    1 & \text{if $v \in s(e)$} \\
    0 & \text{otherwise}
  \end{cases}
  = {(A_s)}_{ve} = {(A_s^T)}_{ev}
\]
for all $e \in E$ and $v \in V$,
such that $A_{s^*} = A_s^T = A_s^*$ and similarly $A_{r^*} = A_r^*$.
Thus, the $*$-isomorphism $\varphi$ exists and it is a morphism
of compact quantum groups, since
\[
  \Delta(\varphi(u_{v w})) 
  = \sum_{x \in V} \widehat{u}_{v x} \widehat{u}_{x w}
  = \sum_{x \in V} \varphi(u_{v x}) \varphi(u_{x w})
  = (\varphi \otimes \varphi)( \Delta(u_{v w}) )
\]
for all $v, w \in V$ and $\Delta(\varphi(u_{e f})) = (\varphi \otimes \varphi)( \Delta(u_{e f}) )$ for all $e, f \in E$ by the same computation.
\end{proof}

\subsection{Hypergraphs without multi-edges}

In contrast to the quantum automorphism group of graphs by Bichon and Banica,
our quantum automorphism group includes an additional magic unitary $u_E$ for the edges.
This magic unitary is necessary to capture quantum symmetries between multi-edges, see for example the 
family of hypergraphs in \Cref{sec:example-max-qsym}. However, it turns out that if a hypergraph has no multi-edges,
then the magic unitary $u_E$ is redundant and we can express the entries of $u_E$ in terms of the entries of $u_V$.

We begin with the following lemma, which relates the entries 
of $u_E$ with the entries of $u_V$.

\begin{lemma}\label{lem:rep-edges-1}
Let $\Gamma := (V, E)$ be a hypergraph and $X \subseteq V$. Then
\[
  \sum_{\substack{f \in E \\ X \subseteq s(f)}} u_{fe}
  = \prod_{v \in X} \sum_{\substack{w \in V \\ w \in s(e)}} u_{vw}
  \quad
  \forall e \in E.
\]
In particular, the product commutes. The same statement also holds for the range map $r$.
\end{lemma}
\begin{proof}
Let $e \in E$ and $X = \{v_1, \ldots, v_k \}$, where all $v_i$ are distinct. Then 
\[
  \prod_{v \in X} {(A_s u_E)}_{v e}
  = \prod_{v \in X} \sum_{\substack{f \in E \\ v \in s(f)}} u_{fe}
  = \sum_{\substack{f_1 \in E \\ v_1 \in s(f_1)}} \cdots \sum_{\substack{f_k \in E \\ v_k \in s(f_k)}} u_{f_1 e} \cdots u_{f_k e}.
\]
Since 
\[
  u_{f_1 e} \cdots u_{f_k e} = \delta_{f_1 f_2} \cdots \delta_{f_1 f_k} u_{f_1 e},
\]
it follows that 
\[
  \prod_{v \in X} {(A_s u_E)}_{v e} 
  = \sum_{\substack{f \in E \\ v_1 \in s(f), \ldots, v_k \in s(f)}} u_{f e}
  = \sum_{\substack{f \in E \\ X \subseteq s(f)}} u_{fe}.
\]
On the other hand, we can apply $A_s u_E = u_V A_s$ to the original expression to obtain
\[
  \prod_{v \in X} {(A_s u_E)}_{v e} 
  = \prod_{v \in X} {(u_V A_s)}_{v e}
  = \prod_{v \in X} \sum_{\substack{w \in V \\ w \in s(e)}} u_{vw}.
\]
Thus, 
\[
  \sum_{\substack{f \in E \\ X \subseteq s(f)}} u_{fe}
  = \prod_{v \in X} \sum_{\substack{w \in V \\ w \in s(e)}} u_{vw}.
\]
Since the left side does not depend on the order of $X$, the product on the right side commutes.
Further, we can replace $s$ with $r$ to obtain a proof of the corresponding statement for the range map $r$.
\end{proof}

By using an inclusion-exclusion argument, we can further 
improve the previous lemma to obtain the equality $X = s(f)$
instead of $X \subseteq s(f)$ on the left hand side.

\begin{lemma}\label{lem:rep-edges-inc-exc}
Let $\Gamma := (V, E)$ be a hypergraph and $X \subseteq V$. Then
\[
  \sum_{\substack{f \in E \\ s(f) = X}} u_{fe} 
  = \sum_{X \subseteq Y \subseteq V} {(-1)}^{\abs{Y} - \abs{X}} \prod_{v \in Y} \sum_{\substack{w \in V \\ w \in s(e)}}  u_{vw}
  \quad
  \forall e \in E.
\]
The same statement also holds for the range map $r$.
\end{lemma}
\begin{proof}
By rewriting the statement using \Cref{lem:rep-edges-1}, we have to show that 
\[
  \sum_{\substack{f \in E \\ s(f) = X}} u_{fe} 
  = 
  \sum_{X \subseteq Y \subseteq V} {(-1)}^{\abs{Y} - \abs{X}}
  \sum_{\substack{f \in E \\ Y \subseteq s(f)}} u_{fe}
  \quad 
  \forall e \in E.
\]
Now, consider an element $u_{fe}$ with $X \subseteq s(f)$
and define $k := |s(f)| - |X|$.
Then there are exactly
$\binom{k}{\ell}$ subsets $Y$ with $X \subseteq Y \subseteq s(f)$ and 
$|X| + \ell$ elements. Further, each is weighted with a 
factor of 
\[
  {(-1)}^{|Y|-|X|} = {(-1)}^{(|X| + \ell) - |X|} = {(-1)}^{\ell}
\]
on the right side of the equation. Thus, by the binomial theorem, the total
contribution of $u_{ef}$ on the right side is
\[
  \sum_{\ell = 0}^k \binom{k}{\ell} {(-1)}^{\ell} 
  = {((-1) + 1)}^k 
  = \begin{cases}
    1 & \text{if $k = 0$}, \\
    0 & \text{if $k > 0$}.
  \end{cases}
\]
Therefore, the right side of the 
equation contains exactly each $u_{fe}$ with $k = 0$, which 
is equivalent to $s(f) = X$. 
The corresponding statement for the range map $r$ can be proven in the same way.
\end{proof}

The previous lemma can now be used to show that 
the elements of $u_E$ can be expressed in terms of the elements of $u_V$ 
if the underlying hypergraph has no multi-edges.

\begin{theorem}\label{thm:gens-no-multi-edges}
Let $\Gamma:= (V, E)$ be a
hypergraph without multi-edges. 
Denote with $C^*(u_V)$ the $C^*$-algebra generated by $u_{vw}$ for all $v, w\in V$.
Then 
\[
  u_{ef} \in C^*(u_V) \quad \forall e,f \in E.
\]
\end{theorem}
\begin{proof}
Let $e, f \in E$. Then 
\[
  \sum_{\substack{g \in E \\ s(g) = s(e)}} u_{gf},
  \sum_{\substack{g \in E \\ r(g) = r(e)}} u_{gf}
  \in C^*(u_V)
\]
by chosing $X = s(e)$ and $X = r(e)$ in \Cref{lem:rep-edges-inc-exc}. This implies
\[
  \bigg( \sum_{\substack{g \in E \\ s(g) = s(e)}} u_{gf} \bigg)
  \bigg( \sum_{\substack{g \in E \\ r(g) = r(e)}} u_{gf} \bigg)
  =
  \sum_{\substack{g_1, g_2 \in E \\ s(g_1) = s(e) \\ r(g_2) = r(e)}}
   \underbrace{u_{g_1 e} u_{g_2 e}}_{\delta_{g_1 g_2} u_{g_1 f}} 
  = 
  \sum_{\substack{g \in E \\ s(g) = s(e) \\ r(g) = r(e)}} u_{gf} 
  \in C^*(u_V).
\]
However, $\Gamma$ has no multi-edges, such that 
\[
  u_{ef} = \sum_{\substack{g \in E \\ s(g) = s(e) \\ r(g) = r(e)}} u_{gf} 
  \in C^*(u_V).
\]
 
\end{proof}

Translating the previous theorem to the setting of quantum groups yields the following two corollaries.

\begin{corollary}\label{corr:subgroup-SV}
Let $\Gamma:= (V, E)$ be a hypergraph without multi-edges. Then 
\[
  \Aut^+(\Gamma) \subseteq S_V^+.
\]
\end{corollary}
\begin{proof}
Denote with $\textit{u}$ the fundamental representation of $\Aut^+(\Gamma)$
and with $\widehat{u}$ the fundamental representation of $S_V^+$.
As in the proof of \Cref{prop:free-prod-subgroup}, the exists a morphism
of compact quantum groups $\varphi \colon C(S_V^+) \to \Aut^+(\Gamma)$, 
\[
    \widehat{u}_{vw} \mapsto u_{vw} \quad \forall v, w \in V.
\]
By \Cref{thm:gens-no-multi-edges}, this morphism is surjective, such that 
$\Aut^+(\Gamma) \subseteq S_V^+$.
\end{proof}

\begin{corollary}\label{corr:subgroup-SE}
Let $\Gamma:= (V, E)$ be a hypergraph such that $\Gamma^*$ has no multi-edges. Then 
\[
  \Aut^+(\Gamma) \subseteq S_E^+.
\]
\end{corollary}
\begin{proof}
By combining \Cref{prop:dual-isom} with \Cref{corr:subgroup-SV}, we obtain 
\[
  \Aut^+(\Gamma) = \Aut^+(\Gamma^*) \subseteq S_E^+.
\]
\end{proof}

\section{Link to quantum symmetries of classical graphs}\label{sec:qsym-graphs}

In this section, we study the quantum
automorphism group $\Aut^+(\Gamma)$ for hypergraphs
which come from a classical directed, simple or multigraph
as defined in \Cref{sec:prelim-graphs-hypergraphs}.
In particular, we show that in this case our quantum automorphism group for hypergraphs
agrees with the quantum automorphism group of Bichon for classical graphs
or its multigraph version by Goswami-Hossain.
Thus, we can view our quantum automorphism group for hypergraphs as a generalization 
of the quantum automorphism group by Bichon for graphs.

\subsection{Directed graphs}

We begin with directed graphs as in \Cref{def:directed-graph}. 
Recall from \Cref{def:dir-graph-to-hypergraph} that we can identify a directed graph $\Gamma := (V, E)$ with 
a $1$-uniform hypergraph without multi-edges by defining the source and range maps
\[
  s(v, w) := \{ v \}, \quad r(v, w) = \{ w \}  
  \quad \forall (v, w) \in E.
\]
In this way, we can apply \Cref{def:hypergraph-qaut} to obtain a hypergraph quantum 
automorphism group $\Aut^+(\Gamma)$. 
In the following, we show that $\Aut^+(\Gamma)$ agrees with the quantum automorphism group 
$\AutBichon(\Gamma)$ by Bichon.

We begin by reformulating the intertwiner relations $A_s u_E = u_V A_s$ and $A_r u_E = u_V A_r$
in the setting of directed graphs.

\begin{lemma}\label{lem:directed-graph-rel}
Let $\Gamma := (V, E)$ be a directed graph, 
$\A$ a unital $C^*$-algebra and
\[
  u_V \in \A \otimes \C^{V \times V},
  \quad 
  u_E \in \A \otimes \C^{E \times E}.
\]
Then the relations 
$A_s u_E = u_V A_s$ and $A_r u_E = u_V A_r$ are equivalent to
\[
  \sum_{\substack{(v_2, w_2) \in E \\ v_0 = v_2}} u_{(v_2, w_2)(v_1, w_1)} = u_{v_0 v_1},
  \qquad
  \sum_{\substack{(v_2, w_2) \in E \\ v_0 = w_2}} u_{(v_2, w_2)(v_1, w_1)} = u_{v_0 w_1}
\]
for all $v_0 \in V$ and $(v_1, w_1) \in E$.
\end{lemma}
\begin{proof}
Let $v_0 \in V$, $e := (v_1, w_1) \in E$. Then
\begin{align*}
  {(A_s u_E)}_{v_0 e} &= \sum_{\substack{f \in E \\ v_0 \in s(f)}} u_{fe} = \sum_{\substack{(v_2, w_2) \in E \\ v_0 = v_2}} u_{(v_2, w_2)(v_1, w_1)}, \\
  {(A_r u_E)}_{v_0 e} &= \sum_{\substack{f \in E \\ v_0 \in r(f)}} u_{fe} = \sum_{\substack{(v_2, w_2) \in E \\ v_0 = w_2}} u_{(v_2, w_2)(v_1, w_1)}.
\end{align*}
On the other hand, the image of $s$ and $r$ always contains one element
and there are no multi-edges, 
such that
\begin{align*}
  {(u_V A_s)}_{v_0 e} &= \sum_{\substack{w \in V \\ w \in s(e)}} u_{v_0 w} = u_{v_0 v_1},\\
  {(u_V A_r)}_{v_0 e} &= \sum_{\substack{w \in V \\ w \in r(e)}} u_{v_0 w} = u_{v_0 w_1}.
\end{align*}
Thus, $A_s u_E = u_V A_s$ and $A_r u_E = u_V A_r$ are equivalent to 
\[
  \sum_{\substack{(v_2, w_2) \in E \\ v_0 = v_2}} u_{(v_2, w_2)(v_1, w_1)} = u_{v_0 v_1},
  \qquad
  \sum_{\substack{(v_2, w_2) \in E \\ v_0 = w_2}} u_{(v_2, w_2)(v_1, w_1)} = u_{v_0 w_1}.
\]
for all $v_0 \in V$ and $(v_1, w_1) \in E$.
\end{proof}

Using the previous lemma, we can now express the entries 
of $u_E$ in terms of the entries of $u_V$.

\begin{lemma}\label{lem:bichon-E}
Let $\Gamma := (V, E)$ be a directed graph and denote
with $u$ the fundamental representation of $\Aut^+(\Gamma)$. Then
\[
  u_{(v_1, w_1) (v_2, w_2)} 
  = u_{v_1 v_2} u_{w_1 w_2} 
  = u_{w_1 w_2} u_{v_1 v_2}
  \quad 
  \forall (v_1, w_1), (v_2, w_2) \in E.
\]
\end{lemma}
\begin{proof}
Let $(v_1, w_1), (v_2, w_2) \in E$. By \Cref{lem:directed-graph-rel}, we have 
\[
  u_{v_1 v_2} = \sum_{\substack{(v_3, w_3) \in E \\ v_1 = v_3}} u_{(v_3, w_3)(v_2, w_2)}, \qquad 
  u_{w_1 w_2} = \sum_{\substack{(v_4, w_4) \in E \\ w_1 = w_4}} u_{(v_4, w_4)(v_2, w_2)}
\]
which yields
\begin{align*}
  u_{v_1 v_2} u_{w_1 w_2} 
  &= 
  \sum_{\substack{(v_3, w_3) \in E \\ v_1 = v_3}} 
  \sum_{\substack{(v_4, w_4) \in E \\ w_1 = w_4}} 
  \underbrace{u_{(v_3, w_3)(v_2, w_2)} u_{(v_4, w_4)(v_2, w_2)}}_{\delta_{(v_3, w_3)(v_4, w_4)} u_{(v_3, w_3)(v_2, w_2)}}
  = 
  \sum_{\substack{(v_3, w_3) \in E \\ v_1 = v_3 \\ w_1 = w_3}} 
  u_{(v_3, w_3)(v_2, w_2)},
  \\
  u_{w_1 w_2} u_{v_1 v_2} 
  &= 
  \sum_{\substack{(v_4, w_4) \in E \\ w_1 = w_4}} 
  \sum_{\substack{(v_3, w_3) \in E \\ v_1 = v_3}} 
  \underbrace{ u_{(v_4, w_4)(v_2, w_2)} u_{(v_3, w_3)(v_2, w_2)}}_{\delta_{(v_4, w_4)(v_3, w_3)} u_{(v_3, w_3)(v_2, w_2)}} 
  = 
  \sum_{\substack{(v_3, w_3) \in E \\ v_1 = v_3 \\ w_1 = w_3}} 
  u_{(v_3, w_3)(v_2, w_2)}.
\end{align*}
Since $\Gamma$ has no multi-edges, we have
\[
  u_{(v_1, w_1) (v_2, w_2)}
  = 
  \sum_{\substack{(v_3, w_3) \in E \\ v_1 = v_3 \\ w_1 = w_3}} 
  u_{(v_3, w_3)(v_2, w_2)}.
\]
Thus,
\[
  u_{(v_1, w_1) (v_2, w_2)} = u_{v_1 v_2} u_{w_1 w_2} = u_{w_1 w_2} u_{v_1 v_2}.
\]
\end{proof}

Now, we can show that our quantum automorphism group $\Aut^+(\Gamma)$ agrees 
with the quantum automorphism group  $\AutBichon(\Gamma)$ when we identify
the magic unitary $u_V$ with the fundamental representation of $\AutBichon(\Gamma)$.
  
\begin{theorem}\label{thm:directed-graph-bichon}
Let $\Gamma$ be a directed graph. Then $\Aut^+(\Gamma) = \AutBichon(\Gamma)$.
\end{theorem}
\begin{proof}
Let $\Gamma := (V, E)$. Denote with $u$ denotes fundamental representation of $\Aut^+(\Gamma)$ and 
$\widehat{u}$ denotes the fundamental representation of $\AutBichon(\Gamma)$.
Further, define the elements
\[
  \widehat{u}_{(v_1, w_1)(v_2, w_2)} := \widehat{u}_{v_1 v_2} \widehat{u}_{w_1 w_2}
  \quad 
  \forall (v_1, w_1), (v_2, w_2) \in E
\]
and the matrices $\widehat{u}_V := {(\widehat{u}_{vw})}_{v,w \in V}$ and 
$\widehat{u}_V := {(\widehat{u}_{ef})}_{e,f \in V}$. 
To prove the statement,
we begin by constructing a unital $*$-homomorphism 
\begin{alignat*}{3}
  \varphi \colon C(\Aut^+(\Gamma)) &\to C(\AutBichon(\Gamma)), &&\\
  u_{vw} &\mapsto \widehat{u}_{vw} &\quad& \forall v,w\in V, \\
  u_{(v_1,w_1)(v_2,w_2)} &\mapsto \widehat{u}_{v_1 v_2}\widehat{u}_{w_1 w_2}
  &\quad& \forall (v_1,w_1),(v_2,w_2)\in E
\end{alignat*}
using the universal property of $C(\Aut^+(\Gamma))$. 
Thus, we have to show that the matrices $\widehat{u}_V$ 
and $\widehat{u}_E$ satisfy the relations from \Cref{def:hypergraph-qaut}.
By \Cref{def:bichon-qaut}, $\widehat{u}_V$ is a magic unitary and we have
\[
  \widehat{u}_{v_1 v_2} \widehat{u}_{w_1 w_2}
  = 
  \widehat{u}_{w_1 w_2} \widehat{u}_{v_1 v_2}
  \quad 
  \forall (v_1, w_1),(v_2, w_2) \in E,
\]
which implies
\begin{gather*}
  {(\widehat{u}_{v_1 v_2} \widehat{u}_{w_1 w_2})}^* 
  = {(\widehat{u}_{w_1 w_2})}^* {(\widehat{u}_{v_1 v_2})}^* 
  = \widehat{u}_{w_1 w_2} \widehat{u}_{v_1 v_2}
  = \widehat{u}_{v_1 v_2} \widehat{u}_{w_1 w_1},
  \\
  {(\widehat{u}_{v_1 v_2} \widehat{u}_{w_1 w_2})}^2 
  = \widehat{u}_{v_1 v_2} \widehat{u}_{w_1 w_2} \widehat{u}_{v_1 v_2} \widehat{u}_{w_1 w_2} 
  = {(\widehat{u}_{v_1 v_2})}^2 {(\widehat{u}_{w_1 w_2})}^2 
  = \widehat{u}_{v_1 v_2} \widehat{u}_{w_1 w_2}.
\end{gather*}
Further, the additional relations from \Cref{prop:bic-sum-one} yield
\begin{align*}
  & \sum_{(v_2, w_2) \in E} u_{v_1 v_2} u_{w_1 w_2} 
  = \sum_{\substack{v_2, w_2 \in V \\ (v_2, w_2) \in E}} u_{v_1 v_2} u_{w_1 w_2} 
  = 1, \\  
  & \sum_{(v_2, w_2) \in E} u_{v_2 v_1} u_{w_2 w_1} 
  = \sum_{\substack{v_2, w_2 \in V \\ (v_2, w_2) \in E}} u_{v_2 v_1} u_{w_2 w_1} 
  = 1
\end{align*}
for all $(v_1, w_1) \in E$. Hence, $\widehat{u}_E$ is a magic unitary.
Next, we check the intertwiner relation $A_s \widehat{u}_E = \widehat{u}_V A_s$.
Let $v_0 \in V$ and $(v_1, w_1) \in E$. Then
\[
  \sum_{\substack{(v_2, w_2) \in E \\ v_0 = v_2}}
  \widehat{u}_{v_2 v_1} \widehat{u}_{w_2 w_1} \\
  =
  \widehat{u}_{v_0 v_1} 
  \sum_{\substack{(v_2, w_2) \in E \\ v_0 = v_2}}
  \widehat{u}_{w_2 w_1} \\
  =
  \widehat{u}_{v_0 v_1} 
  \sum_{w_2 \in V}
  {(\Adj)}_{v_0 w_2 }\widehat{u}_{w_2 w_1}.
\]
By \Cref{def:bichon-qaut}, we have $\Adj \widehat{u}_V = \widehat{u}_V \Adj$, such that
\[
  \widehat{u}_{v_0 v_1} 
  \sum_{w_2 \in V}
  {(\Adj)}_{v_0 w_2 }\widehat{u}_{w_2 w_1}
  =
  \widehat{u}_{v_0 v_1} 
  \sum_{w_2 \in V}
  \widehat{u}_{v_0 w_2 } {(\Adj)}_{w_2 w_1} 
\]
Since $\widehat{u}_{v_0 v_1} \widehat{u}_{v_0 w_2} = \delta_{v_1 w_2} \widehat{u}_{v_0 v_1}$, it follows that 
\[
  \widehat{u}_{v_0 v_1} 
  \sum_{w_2 \in V}
  \widehat{u}_{v_0 w_2 } {(\Adj)}_{w_2 w_1} 
  =
  \sum_{w_2 \in V}
  \delta_{v_1 w_2} \widehat{u}_{v_0 v_1} {(\Adj)}_{w_2 w_1} 
  = \widehat{u}_{v_0 v_1} {(\Adj)}_{v_1 w_1} 
  = \widehat{u}_{v_0 v_1}.
\]
Thus, we have in total
\[
  \sum_{\substack{(v_2, w_2) \in E \\ v_0 = v_2}}
  \widehat{u}_{(v_2, w_2)(v_1, w_1)}
  =
  \sum_{\substack{(v_2, w_2) \in E \\ v_0 = v_2}}
  \widehat{u}_{v_2 v_1} \widehat{u}_{w_2 w_1}
  =
  \widehat{u}_{v_0 v_1}
  ,
\]
which is equivalent to $A_s \widehat{u}_E = \widehat{u}_V A_s$ by \Cref{lem:directed-graph-rel}.
Similarly, one shows $A_r \widehat{u}_E = \widehat{u}_V A_r$, such that the map $\varphi$ exists.

Next, we construct the inverse map 
\[
  \psi \colon C(\AutBichon(\Gamma)) \to C(\Aut^+(\Gamma)), \quad \widehat{u}_{vw} \mapsto u_{vw}
  \quad 
  \forall v, w \in V
\]
using the universal property of $C(\AutBichon(\Gamma))$. Thus, we have to show that the matrix $u_V$ 
satisfies the relations from \Cref{def:bichon-qaut}.
First, note that $u_V$ is a magic unitary by \Cref{def:hypergraph-qaut}.
Secondly, we have to show $\Adj u_V = u_V \Adj$. Observe that $\Adj = A_s A_r^*$, since 
\[
  {(A_s A_r^*)}_{v w} 
  = \sum_{e \in E} {(A_s)}_{v e} {(A_r)}_{w e}
  = \sum_{e \in E} \delta_{(v, w)e}
  = \begin{cases}
    1 & \text{if $(v, w) \in E$}, \\
    0 & \text{otherwise},
  \end{cases}
\]
for all $v, w \in V$. Hence,
\[
  \Adj u_V = A_s A_r^* u_V = A_s u_E A_r^* = u_V A_s A_r^* = u_V \Adj,
\]
because $A_r^*$ also intertwines $u_V$ and $u_E$ by \Cref{rem:hypergraph-qaut-rels}.
Finally, we have to check that
\[
  u_{v_1 v_2} u_{w_1 w_2} = u_{w_1 w_2} u_{v_1 v_2} \quad \forall (v_1, w_1), (v_2, w_2) \in E.
\]
But this follows directly from \Cref{lem:bichon-E}, since 
\[
  u_{v_1 v_2} u_{w_1 w_2} = u_{(v_1, w_1)(v_2, w_2)} = u_{w_1 w_2} u_{v_1 v_2} .
\]
Thus, the $*$-homomorphism $\psi$ exists.

Note that the maps $\varphi$ and $\psi$ are indeed inverse, because
\begin{alignat*}{3}
  \widehat{u}_{v w}
  & \ \longleftrightarrow \
  u_{v w} \\
  \widehat{u}_{v_1 v_2}\widehat{u}_{w_1 w_2} 
  & \ \longleftrightarrow \ 
  u_{v_1 v_2} u_{w_1 w_2} = u_{(v_1, w_1)(v_2, w_2)}
\end{alignat*}
by \Cref{lem:bichon-E}. 
Thus, remains to show that $\psi$ respects the comultiplication and defines an isomorphism of 
compact quantum groups. But this follows directly, since 
\[
  \Delta(\psi(\widehat{u}_{v w}))
  = \sum_{x \in V} u_{v x} \otimes u_{x w}
  = \sum_{x \in V} \psi(\widehat{u}_{v x}) \otimes \psi(\widehat{u}_{x w})
  = (\psi \otimes \psi)(\Delta( \widehat{u}_{v w}))
\]
for all $v, w \in V$.
\end{proof}

\subsection{Simple graphs}
Next, we come to simple graphs as in \Cref{def:simple-graph}.
Recall from \Cref{def:simple-graph-to-hypergraph} that we 
can regard simple graphs as $2$-uniform undirected hypergraphs without multi-edges by 
defining the source and range maps
\[
  s(\{v, w\}) := \{ v, w \},
  \qquad 
  r(\{v, w\}) := \{ v, w \}
  \qquad 
  \forall \{v, w\} \in E.
\]
In the following, we show that for a simple graph $\Gamma$ our 
quantum automorphism group $\Aut^+(\Gamma)$ agrees again
with the quantum automorphism group $\AutBichon(\Gamma)$ by Bichon.

We begin with a reformulation of the intertwiner relations
$A_s u_E = u_V A_s$ and $A_r u_E = u_V A_r$ in the setting of simple graphs.

\begin{lemma}\label{lem:simple-graph-inter}
Let $\Gamma := (V, E)$ be a simple graph, 
$\A$ a unital $C^*$-algebra and
\[
  u_V \in \A \otimes \C^{V \times V},
  \quad 
  u_E \in \A \otimes \C^{E \times E}
\]
two unitaries. Then the following
are equivalent:
\begin{enumerate}
\item $A_s u_E = u_V A_s$ and $A_r u_E = u_V A_r$,
\item $\displaystyle
  \sum_{\substack{v_2 \in V \\ \{v_0, v_2\} \in E}}
  u_{\{v_0, v_2\}\{v_1, w_1\}}
  = u_{v_0 v_1} + u_{v_0 w_1}
  \quad 
  \forall v_0 \in V, \, \{v_1, w_1\} \in E,
$
\item $\displaystyle
  \sum_{\substack{v_2 \in V \\ \{ v_0, v_2 \} \in E }} u_{\{ v_1, w_1 \} \{v_0, v_2\}} 
  = u_{v_1 v_0} + u_{w_1 v_0}
  \quad 
  \forall v_0 \in V, \, \{v_1, w_1\} \in E.
$
\end{enumerate}
\end{lemma}
\begin{proof}
Since $A_s = A_r$, we only have to consider $A_s u_E = u_V A_s$.
Let $v_0 \in V$ and $e := \{ v_1, w_1 \} \in E$. Then 
\begin{align*}
  {(A_s u_E)}_{v_0 e}
  &= 
  \sum_{f \in E} {(A_s)}_{v_0 f} u_{f e}
  =
  \sum_{\substack{f \in E \\ v_0 \in f}} u_{f e} 
  =
  \sum_{\substack{v_2 \in V \\ \{ v_0, v_2 \} \in E }} u_{\{v_0, v_2\} \{ v_1, w_1 \}}, \\
  {(u_V A_s)}_{v_0 e}
  &=
  \sum_{w \in V} u_{v_0 w} {(A_s)}_{w e}
  =
  \sum_{\substack{w \in V \\ w \in e}} u_{v_0 w}
  = u_{v_0 v_1} + u_{v_0 w_1}.
\end{align*}
Hence, $A_s u_E = u_V A_s$ is equivalent to 
\[
  \sum_{\substack{v_2 \in V \\ \{ v_0, v_2 \} \in E }} u_{\{v_0, v_2\} \{ v_1, w_1 \}} = u_{v_0 v_1} + u_{v_0 w_1}
  \quad \forall v_0 \in V, \, \{v_1, w_1\} \in E.
\]
Similarly, we compute 
\begin{align*}
  {(u_E A_s^*)}_{e v_0}
  &= 
  \sum_{f \in E} u_{e f} {(A_s^*)}_{f v_0} 
  =
  \sum_{\substack{f \in E \\ v_0 \in f}} u_{e f} 
  =
  \sum_{\substack{v_2 \in V \\ \{ v_0, v_2 \} \in E }} u_{\{ v_1, w_1 \} \{v_0, v_2\}}, \\
  {(A_s^* u_V )}_{e v_0}
  &=
  \sum_{w \in V} {(A_s)}_{e w} u_{w v_0} 
  =
  \sum_{\substack{w \in V \\ w \in e}} u_{w v_0}
  = u_{v_1 v_0} + u_{w_1 v_0}.
\end{align*}
Thus, $u_E A_s^* = A_s^* u_V$ is equivalent to 
\[
  \sum_{\substack{v_2 \in V \\ \{ v_0, v_2 \} \in E }} u_{\{ v_1, w_1 \} \{v_0, v_2\}} = u_{w_1 v_1} + u_{w_2 v_1}
  \quad 
  \forall v_0 \in V, \, \{v_1, w_1\} \in E,
\]
which is again equivalent to $A_s u_E = u_V A_s$ by \Cref{prop:star-intertwiner}.
\end{proof}

As in the case of directed graphs, we can now use the previous lemma to express the 
entries of the magic unitary $u_E$ in terms of the entries of $u_V$.

\begin{lemma}\label{lem:simple-graph-u-E}
Let $\Gamma := (V, E)$ be a simple graph and 
denote with $u$ the fundamental representation
of $\Aut^+(\Gamma)$. Then
\begin{align*}
  u_{\{v_1, w_1 \}\{v_2,w_2\}}
  &= u_{v_1 v_2} u_{w_1 w_2} + u_{v_1 w_2} u_{w_1 v_2 } \\
  &= u_{v_1 v_2} u_{w_1 w_2} + u_{w_1 v_2} u_{v_1 w_2}
\end{align*}
for all $\{v_1, w_1\}, \{v_2, w_2\} \in E$.
\end{lemma}
\begin{proof}
Let $\{v_1, w_1\}, \{v_2, w_2\} \in E$. Since $v_1 \neq w_1$, we have
\begin{align*}
(u_{v_1 v_2} + u_{v_1 w_2})(u_{w_1 v_2} + u_{w_1 w_2})
&= 
\underbrace{u_{v_1 v_2} u_{w_1 v_2}}_{0} + u_{v_1 v_2} u_{w_1 w_2}
+ u_{v_1 w_2} u_{w_1 v_2} + \underbrace{u_{v_1 w_2} u_{w_1 w_2}}_{0} \\
&=
u_{v_1 v_2} u_{w_1 w_2} + u_{v_1 w_2} u_{w_1 v_2}.
\end{align*}
On the other hand, \Cref{lem:simple-graph-inter} yields
\begin{align*}
(u_{v_1 v_2} + u_{v_1 w_2})(u_{w_1 v_2} + u_{w_1 w_2})
&=
\sum_{\substack{v_3 \in V \\ \{v_1, v_3\} \in E}}
  u_{\{v_1, v_3\}\{v_2, w_2\}}
\sum_{\substack{v_4 \in V \\ \{w_1, v_4\} \in E}}
  u_{\{w_1, v_4\}\{v_2, w_2\}} \\
&= 
\sum_{\substack{v_3, v_4 \in V \\ \{v_1, v_3\}, \{w_1, v_4\} \in E}}
\underbrace{u_{\{v_1, v_3\}\{v_2, w_2\}} u_{\{w_1, v_4\}\{v_2, w_2\}}}_{\delta_{\{v_1, v_3\}\{w_1, v_4\}} u_{\{v_1, v_3\}\{v_2, w_2\}} } \\
&= 
u_{\{v_1, w_1\}\{v_2, w_2\}},
\end{align*}
where $\{v_1, v_3\} = \{w_1, v_4\}$ and $v_1 \neq w_1$ implies 
$v_3 = w_1$ in the last step.
Thus,
\[
  u_{\{v_1, w_1\}\{v_2, w_2\}} 
  = (u_{v_1 v_2} + u_{v_1 w_2})(u_{w_1 v_2} + u_{w_1 w_2}) 
  = u_{v_1 v_2} u_{w_1 w_2} + u_{v_1 w_2} u_{w_1 v_2}. 
\]
Similarly, we compute
\[
  (u_{v_1 v_2} + u_{w_1 v_2})
  (u_{v_1 w_2} + u_{w_1 w_2})
  =
  u_{v_1 v_2} u_{w_1 w_2}
  + 
  u_{w_1 v_2} u_{v_1 w_2}
\]
and 
\begin{align*}
  (u_{v_1 v_2} + u_{w_1 v_2})
  (u_{v_1 w_2} + u_{w_1 w_2})
  &=
  \sum_{\substack{v_3, v_4 \in V \\ \{v_1, v_3\}, \{v_2, v_4\} \in E}}
  \underbrace{u_{\{v_1, w_1\} \{v_2, v_3\} } u_{\{v_1, w_1\} \{w_2, v_4\}}}_{\delta_{\{v_2, v_3\} \{w_2, v_4\}} u_{\{v_1, w_1\} \{v_2, v_3\} }} \\
  &=
  u_{\{v_1, w_1\} \{v_2, w_2 \}}.
\end{align*}
using $v_2 \neq w_2$ and the other part of \Cref{lem:simple-graph-inter}. Hence,
\[
  u_{\{v_1, w_1\} \{v_2, w_2 \}}
  =
  u_{v_1 v_2} u_{w_1 w_2}
  + 
  u_{w_1 v_2} u_{v_1 w_2}
  \quad
  \forall \{v_1, w_1\}, \{v_2, w_2\} \in E.
\]
\end{proof}

In addition to the previous lemmas, we need the following proposition for general hypergraphs, which states that 
$u_{vw} = 0$ if the vertices $v$ and $w$ are contained in a different number of edges. This proposition
will also be used in \Cref{sec:qaut-multigraphs}
when computing the quantum automorphism groups of multigraphs.

Note that this type of relation seems to be useful when computing quantum 
symmetries of concrete hypergraphs. For example, the similar relation $u_{ij} u_{kl} = 0$ for $d(i, k) \neq d(j, l)$ 
was used in~\cite{schmidt18-2, schmidt20-2} to compute quantum symmetries of classical graphs, where $d$ denotes the distance between two vertices.

\begin{proposition}\label{prop:N-intertwiner}
Let $\Gamma := (V, E)$ be a hypergraph and 
denote with $u$ the fundamental representation
of $\Aut^+(\Gamma)$. Define
\[
  N_s(v) := \abs{\{ e \in E ~|~ v \in s(e) \}}
  \quad \forall v \in V.
\]
Then 
\[
  N_s(v) \cdot u_{vw} = u_{vw} \cdot N_s(w)
  \quad \forall v, w \in V.
\]
In particular, $N_s(v) \neq N_s(w)$ implies $u_{vw} = 0$. The same statement
also holds for the range map $r$.
\end{proposition}
\begin{proof}
Let $v \in V$. By summing both sides of ${(A_s u_E)}_{ve} = {(u_V A_s)}_{ve}$ over all $e \in E$, we obtained
\begin{gather*}
  \sum_{e \in E} \sum_{f \in E} {(A_s)}_{vf} u_{fe}
  = 
  \sum_{f \in E} {(A_s)}_{vf} \underbrace{\bigg(\sum_{e \in E} u_{fe}\bigg)}_{1}
  =
  \sum_{f \in E} {(A_s)}_{vf}
  =
  N_s(v), \\
  \sum_{e \in E} \sum_{w \in V} u_{vw} {(A_s)}_{we}
  =
  \sum_{w \in V} u_{vw} \bigg( \sum_{e \in E} {(A_s)}_{we} \bigg)
  =
  \sum_{w \in V} u_{vw} \cdot N_s(w).
\end{gather*}
Thus,
\[
  N_s(v) = \sum_{w \in V} u_{vw} \cdot N_s(w),
\] 
which implies
\[
  N_s(v) \cdot u_{vw} = u_{vw} \cdot N_s(v) = \sum_{x \in V} \underbrace{u_{vw} u_{vx}}_{\delta_{wx} u_{vw}} \cdot N_s(x)
  = u_{vw} \cdot N_s(w) 
  \quad
  \forall v,w \in V.
\]
In particular, if $N_s(v) \neq N_s(w)$, then
\[
  N_s(v) \cdot u_{vw} = u_{vw} \cdot N_s(w)
  \ \Longleftrightarrow \
  (N_s(v) - N_s(w)) \cdot u_{vw} = 0
  \ \Longleftrightarrow \
  u_{vw} = 0.
\]
The corresponding statement and proof for the range map 
can be obtained by replacing $s$ with $r$.
\end{proof}

We can now show that for simple graphs our quantum automorphism group 
agrees with the quantum automorphism group by Bichon
when we identify the magic unitary $u_V$ with the fundamental representation
of $\AutBichon(\Gamma)$.

\begin{theorem}\label{thm:simple-graph-bichon}
Let $\Gamma$ be a simple graph. Then $\Aut^+(\Gamma) = \AutBichon(\Gamma)$.
\end{theorem}
\begin{proof}
Denote with $u$ the fundamental representation
of $\Aut^+(\Gamma)$ and with $\widehat{u}$ the fundamental
representation of $\AutBichon(\Gamma)$. Further, define 
the elements
\[
  \widehat{u}_{\{v_1, w_1\} \{v_2, w_2\}}
  := 
  \widehat{u}_{v_1 v_2} \widehat{u}_{w_1 w_2} + \widehat{u}_{v_1 w_2} \widehat{u}_{w_1 v_2}
  \quad 
  \forall \{v_1, w_1\}, \{v_2, w_2\} \in E
\]
and the matrices $\widehat{u}_V := {(\widehat{u}_{vw})}_{v,w\in V}$
and $\widehat{u}_E := {(\widehat{u}_{ef})}_{e,f\in E}$.
Note that $\widehat{u}_{\{v_1, w_1\} \{v_2, w_2\}}$ is 
well-defined, since 
\[
  \widehat{u}_{v_1 v_2} \widehat{u}_{w_1 w_2}
  =
  \widehat{u}_{w_1 w_2} \widehat{u}_{v_1 v_2}
  \quad 
  \forall
  \{v_1, w_1\}, \{v_2, w_2\} \in E
\]
by \Cref{def:bichon-qaut}.
First, we construct a unital $*$-homomorphism
\begin{alignat*}{3}
  \varphi \colon C(\Aut^+(\Gamma)) &\to C(\AutBichon(\Gamma)), \\
  u_{vw} & \mapsto \widehat{u}_{vw} & \quad & \forall v, w \in V, \\
  u_{ef} & \mapsto \widehat{u}_{ef} & \quad & \forall e, f \in E 
\end{alignat*}
using the universal property of $C(\Aut^+(\Gamma))$.
Thus, we have to show that $\widehat{u}_V$ and $\widehat{u}_E$
are magic unitaries which are intertwined by $A_s = A_r$.
By \Cref{def:bichon-qaut}, the matrix $\widehat{u}_V$ is a magic unitary
and we can compute
\begin{align*}
  \widehat{u}_{\{v_1, w_1\}, \{v_2, w_2\}}^*
  &= {(\widehat{u}_{v_1 v_2} \widehat{u}_{w_1 w_2})}^* + {(\widehat{u}_{v_1 w_2} \widehat{u}_{w_1 v_2})}^* \\
  &= \widehat{u}_{w_1 w_2} \widehat{u}_{v_1 v_2} + \widehat{u}_{w_1 v_2} \widehat{u}_{v_1 w_2} \\
  &= \widehat{u}_{v_1 v_2} \widehat{u}_{w_1 w_2} + \widehat{u}_{v_1 w_2} \widehat{u}_{w_1 v_2} \\
  &= \widehat{u}_{\{v_1, w_1\}, \{v_2, w_2\}}
\end{align*}
for all $\{v_1, w_1\}, \{v_2, w_2\} \in E$. Similarly, we compute
\begin{align*}
  \widehat{u}_{\{v_1, w_1\}\{v_2, w_2\}}^2
  &=
  (\widehat{u}_{v_1 v_2} \widehat{u}_{w_1 w_2} + \widehat{u}_{v_1 w_2} \widehat{u}_{w_1 v_2})
  (\widehat{u}_{v_1 v_2} \widehat{u}_{w_1 w_2} + \widehat{u}_{v_1 w_2} \widehat{u}_{w_1 v_2}) \\
  &=
  \widehat{u}_{v_1 v_2} \widehat{u}_{w_1 w_2} \widehat{u}_{v_1 v_2} \widehat{u}_{w_1 w_2}
  +
  \widehat{u}_{v_1 w_2} \widehat{u}_{w_1 v_2} \widehat{u}_{v_1 w_2} \widehat{u}_{w_1 v_2} \\
  &=
  \widehat{u}_{v_1 v_2}^2 \widehat{u}_{w_1 w_2}^2 + \widehat{u}_{v_1 w_2}^2 \widehat{u}_{w_1 v_2}^2 \\
  &=
  \widehat{u}_{v_1 v_2} \widehat{u}_{w_1 w_2} + \widehat{u}_{v_1 w_2} \widehat{u}_{w_1 v_2} \\
  &= 
  \widehat{u}_{\{v_1, w_1\}\{v_2, w_2\}}
\end{align*}
for all $\{v_1, w_1\}, \{v_2, w_2\} \in E$, 
where we additionally used the fact that 
\[
  \widehat{u}_{w_1 w_2} \widehat{u}_{v_1 w_2} = 0, \qquad \widehat{u}_{w_1 v_2} \widehat{u}_{v_1 v_2} = 0,
\]
since $v_1 \neq w_1$.
Next, we want to show that the rows and columns of $\widehat{u}_E$ sum to $1$. 
Observe, that 
\[
  \sum_{\{v_1, w_1\} \in E} \widehat{u}_{v_1 v_2} \widehat{u}_{w_1 w_2}
  = 
  \frac{1}{2}
  \sum_{\substack{v_1, w_1 \in V \\ \{v_1, w_1\} \in V}} \widehat{u}_{v_1 v_2} \widehat{u}_{w_1 w_2}
  =
  \frac{1}{2}
  \quad 
  \forall 
  \{v_2, w_2\} \in E
\]
by \Cref{prop:bic-sum-one}. Therefore,
\[
  \sum_{\{v_1, w_1\} \in E} \widehat{u}_{\{v_1, w_1\} \{v_2, w_2\}}
  = 
  \sum_{\{v_1, w_1\} \in E} \widehat{u}_{v_1 v_2} \widehat{u}_{w_1 w_2}
  +
  \sum_{\{v_1, w_1\} \in E} \widehat{u}_{v_1 w_2} \widehat{u}_{w_1 v_2}
  =
  \frac{1}{2} + \frac{1}{2}
  = 
  1
\]
for all $\{v_2, w_2\} \in E$. Similarly, we have 
\[
  \sum_{\{v_2, w_2\} \in E} \widehat{u}_{\{v_1, w_1\} \{v_2, w_2\}}
  = 
  \frac{1}{2} + \frac{1}{2}
  = 
  1
\]
for all $\{v_1, w_1\} \in E$.
Hence, $\widehat{u}_E$ is a magic unitary. It remains to show that 
$A_s \widehat{u}_E = \widehat{u}_V A_s$. But this follows from \Cref{lem:simple-graph-inter}
and $\Adj \widehat{u}_V = \widehat{u}_V \Adj$, 
because
\begin{align*}
  \sum_{\substack{v_2 \in V \\ \{v_0, v_2\} \in E}}
  \widehat{u}_{\{v_0, v_2\}\{v_1, w_1\}} 
  &= 
  \sum_{\substack{v_2 \in V \\ \{v_0, v_2\} \in E}}
  \left( \widehat{u}_{v_0 v_1} \widehat{u}_{v_2 w_1} + \widehat{u}_{v_0 w_1} \widehat{u}_{v_2 v_1} \right) \\
  &=
  \widehat{u}_{v_0 v_1} \sum_{\substack{v_2 \in V \\ \{v_0, v_2\} \in E}} \widehat{u}_{v_2 w_1} 
  + \widehat{u}_{v_0 w_1} \sum_{\substack{v_2 \in V \\ \{v_0, v_2\} \in E}} \widehat{u}_{v_2 v_1} \\
  &=
  \widehat{u}_{v_0 v_1} \sum_{v_2 \in V} {(\Adj)}_{v_0 v_2 } \widehat{u}_{v_2 w_1} 
  + \widehat{u}_{v_0 w_1} \sum_{v_2 \in V} {(\Adj)}_{v_0 v_2} \widehat{u}_{v_2 v_1} \\
  &=
  \widehat{u}_{v_0 v_1} \sum_{v_2 \in V} \widehat{u}_{v_0 v_2 } {(\Adj)}_{v_2 w_1} 
  + \widehat{u}_{v_0 w_1} \sum_{v_2 \in V} \widehat{u}_{v_0 v_2} {(\Adj)}_{v_2 v_1} \\
  &=
  \sum_{v_2 \in V} \underbrace{\widehat{u}_{v_0 v_1} \widehat{u}_{v_0 v_2 }}_{\delta_{v_1 v_2} \widehat{u}_{v_0 v_1}} {(\Adj)}_{v_2 w_1} 
  + \sum_{v_2 \in V} \underbrace{\widehat{u}_{v_0 w_1} \widehat{u}_{v_0 v_2}}_{\delta_{w_1 v_2} \widehat{u}_{v_0 w_1}} {(\Adj)}_{v_2 v_1} \\
  &= 
  \widehat{u}_{v_0 v_1} + \widehat{u}_{v_0 w_1}
\end{align*}
for all $v_0 \in V$ and $\{v_1, w_1\} \in E$.
Hence, the map $\varphi$ exists.

Next, we construct the inverse map
\begin{alignat*}{3}
  \psi \colon C(\AutBichon(\Gamma)) &\to C(\Aut^+(\Gamma)),  \\
  \widehat{u}_{vw} & \mapsto u_{vw} & \quad & \forall v, w \in V
\end{alignat*}
using the universal property of $C(\AutBichon(\Gamma))$. 
Hence, we have to show that $u_V$ satisfies the relations
from \Cref{def:bichon-qaut}. By \Cref{def:hypergraph-qaut},
$u_V$ is a magic unitary. To show that $\Adj u_V = u_V \Adj$, 
observe that 
\[
  {(A_s A_s^*)}_{v w}
  =
  \sum_{e \in E} {(A_s)}_{v e} {(A_s)}_{w e} 
  =
  \abs{\{ e \in E ~|~ v \in s(e) \ \land \ w \in s(e) \}}
  \quad 
  \forall v, w \in V.
\] 
Since $\Gamma$ is a simple graph, each $e \in E$ contains exactly two elements,
such that 
\[
  {(A_s A_s^*)}_{v w} = \begin{cases}
      0 & \text{if $v \neq w$, $\{v, w\} \notin E$}, \\
      1 & \text{if $v \neq w$, $\{v, w\} \in E$}, \\
      N_s(v) & \text{if $v = w$}. \\
    \end{cases}
\]
where 
\[
    N_s(v) := \abs{\{ e \in E ~|~ v \in s(e) \}}
\]
as in \Cref{prop:N-intertwiner}. 
Hence, we can write
\[
    A_s A_s^* = \Adj + T
\]
with $T \in \C^{V \times V}$ defined by 
\[
    T_{vw} := \begin{cases}
      0       & \text{if $v \neq w$}, \\
      N_s(v)  & \text{if $v = w$},
    \end{cases}
    \quad 
    \forall v, w \in V.
\]
\Cref{prop:N-intertwiner}
shows that $T u_V = u_V T$, which implies
\[
    \Adj u_V = 
    A_s A_s^* u_V - T u_V
    =
    A_s u_E A_s^* - u_V T
    =
    u_V A_s A_s^* - u_V T
    =
    u_V \Adj.
\]
Finally, we have to show that 
\[
  u_{v_1 v_2} u_{w_1 w_2} = u_{w_1 w_2} u_{v_1 v_2}
  \quad 
  \forall \{v_1, w_1\}, \{v_2, w_2\} \in E.
\]
But this follows directly from \Cref{lem:simple-graph-u-E}, since 
\begin{align*}
  u_{v_1 v_2} u_{w_1 w_2}  
  &= u_{\{v_1, w_1\}\{v_2,w_2\}} - u_{v_1 w_2} u_{w_1 v_2} \\
  &= u_{\{w_1, v_1\}\{w_2,v_2\}} - u_{v_1 w_2} u_{w_1 v_2} \\
  &= u_{w_1 w_2} u_{v_1 v_2} + u_{v_1 w_2} u_{w_1 v_2} - u_{v_1 w_2} u_{w_1 v_2} \\
  &= u_{w_1 w_2} u_{v_1 v_2}.
\end{align*}
Thus, the $*$-homomorphism $\varphi$ exists. 

By definition of $\varphi$ and $\psi$, we have
\begin{align*}
  u_{vw} \ & \longleftrightarrow \ \widehat{u}_{vw} \\
  u_{\{v_1, w_1\}\{v_2,w_2\}} = u_{v_1 v_2} u_{w_1 w_2} + u_{v_1 w_2} u_{w_1 v_2} \ & \longleftrightarrow \ \widehat{u}_{v_1 v_2} \widehat{u}_{w_1 w_2} + \widehat{u}_{v_1 w_2} \widehat{u}_{w_1 v_2}
\end{align*}
for all $v, w \in V$ and $\{v_1, w_1\}, \{v_2, w_2\} \in E$,
such that both maps are indeed inverse.
Further, $\psi$ is a morphism compact quantum group, since
\[
  \Delta(\psi(\widehat{u}_{v w}))
  = \sum_{x \in V} u_{v x} \otimes u_{x w}
  = \sum_{x \in V} \psi(\widehat{u}_{v x}) \otimes \psi(\widehat{u}_{x w})
  = (\psi \otimes \psi)(\Delta( \widehat{u}_{v w}))
\]
for all $v, w \in V$. Hence, $\psi$ is 
an isomorphism of compact quantum groups, which shows 
$\Aut^+(\Gamma) = \AutBichon(\Gamma)$.
\end{proof}

\subsection{Multigraphs}\label{sec:qaut-multigraphs}

Finally, we consider the case of multigraphs as in \Cref{def:multi-graph}.
Recall from \Cref{def:multigraph-to-hypergraph} that 
we can identify a multigraph $\Gamma := (V, E)$ with source map $s' \colon E \to V$ and range maps $r' \colon E \to V$ 
with a $1$-uniform hypergraphs by defining 
new source and range maps
$s \colon E \to \PP(V)$ and $r \colon E \to \PP(V)$
given by
\[
  s(e) := \{ s'(e) \}, \quad 
  r(e) := \{ r'(e) \} \quad 
  \forall e \in E.
\]

Throughout the rest of this section, we will identify the maps 
$s'$ and $r'$ with the maps $s$ and $r$. In particular, we will write 
$u_{s(e) v}$ instead of $u_{s'(e) v}$ for an edge $e \in E$ and a vertex $v \in V$.

The goal of this section is to show that our quantum automorphisms
group $\Aut^+(\Gamma)$ agrees with 
the quantum automorphism group $\AutGHBichon(\Gamma)$ by Goswami-Hossain 
for multigraphs without isolated vertices.
We begin by reformulating the intertwiner relations $A_s u_E = u_V A_s$ 
and $A_r u_E = u_V A_r$ in the setting of multigraphs.

\begin{lemma}\label{lem:multigraph-intertwiner-rel}
Let $\Gamma := (V, E)$ be a multigraph and denote 
with $u$ the fundamental representation of $\Aut^+(\Gamma)$. Then the relations 
$A_s u_E = u_V A_s$ and $A_r u_E = u_V A_r$ are equivalent to
\[
  \sum_{\substack{f \in E \\ s(f) = v}} u_{e f} = u_{s(e) v},
  \qquad
  \sum_{\substack{f \in E \\ r(f) = v}} u_{e f} = u_{r(e) v}
\]
for all $v \in V$ and $e \in E$.
\end{lemma}
\begin{proof}
First, consider the equation $A_s u_E = u_V A_s$, which is
equivalent to $A_s^* u_V = u_E A_s^*$ by \Cref{prop:star-intertwiner}. A direct computation yields
\[
  {(u_E A^*_s)}_{ev}
  = \sum_{\substack{f \in E \\ s(f) = v}} u_{ef},
  \qquad
  {(A^*_s u_V)}_{ev}
  =
  \sum_{\substack{w \in V \\ s(e) = w}} u_{wv} = u_{s(e) v}
\]
for all $e \in E$ and $v \in V$, such that $A_s^* u_V = u_E A_s^*$ is equivalent to 
\[
  \sum_{\substack{f \in E \\ s(f) = v}} u_{ef} = u_{s(e) v}.
\]
The statement for the range map $r$ can be obtained by replacing $s$ with $r$
in the previous computation.
\end{proof}

To show that our quantum automorphism group agrees with $\AutGHBichon(\Gamma)$, we will identify 
the magic unitary $u_E$ with the fundamental representation of $\AutGHBichon(\Gamma)$.
However, we need further elements in $C(\AutGHBichon(\Gamma))$ which correspond
to the magic unitary $u_V$.

\begin{definition}\label{def:GH-uV}
Let $\Gamma:= (V, E)$ be a multigraph without isolated vertices
and denote with $u$ the fundamental representation of $\AutGHBichon(\Gamma)$.
Define the elements
\[
  u_{vw} := \begin{cases}
    \displaystyle \sum_{\substack{f \in E \\ s(f) = w}} u_{e f} & \text{if there exists $e \in E$ with $s(e) = v$}, \\
    \displaystyle \sum_{\substack{f \in E \\ r(f) = w}} u_{e f} & \text{if there exists $e \in E$ with $r(e) = v$},
  \end{cases}
\]
for all $v, w \in V$.
\end{definition}

Next, we show that these elements $u_{vw}$ are
well-defined and form a quantum permutation on the vertices.

\begin{lemma}\label{lem:GH-uV-magic}
Let $\Gamma:= (V, E)$ be a multigraph without isolated vertices.
Then the elements $u_{vw}$ in \Cref{def:GH-uV} are well-defined. Further, the 
matrix $u_V := {(u_{vw})}_{v,w \in V}$ is a magic unitary.
\end{lemma}
\begin{proof}
First, note that at least one case in \Cref{def:GH-uV} applies to each vertex $v$, since 
$\Gamma$ has no isolated vertices. Further, each case does not depend
on the edge $e$ by Relation~\ref{def:qaut-GH-2} in \Cref{def:qaut-GH}.
To show that overlapping cases are well-defined, assume $v, w \in V$ 
such that $v$ is neither a source nor a sink. If 
$w$ is neither a source nor a sink, then both cases agree by Relation~\ref{def:qaut-GH-4}.
On the other hand, if $w$ is a source or a sink, then both cases yield 
${u}_{vw} = 0$, since the sum is empty in one case, while each ${u}_{ef} = 0$ by 
Relation~\ref{def:qaut-GH-3}
in the other case. Thus, the elements ${u}_{vw}$ are well-defined.

Next, we show that the matrix $u_V := {(u_{vw})}_{v,w \in V}$ is 
a magic unitary. Let $v, w \in V$ and assume without loss of generality 
that there exists an edge $e \in E$ with $v = s(e)$. 
Using \Cref{lem:multigraph-intertwiner-rel} and 
the fact that ${(u_{ef})}_{e,f \in E}$ is a magic unitary, we compute
\begin{align*}
  u_{vw}^*
  &=\sum_{\substack{f \in E \\ s(f) = w}} u_{e f}^*
  =\sum_{\substack{f \in E \\ s(f) = w}} u_{e f}
  = u_{vw}, \\
  u_{vw}^2
  &= \sum_{\substack{f_1 \in E \\ s(f_1) = w}} \sum_{\substack{f_2 \in E \\ s(f_2) = w}} 
  \underbrace{u_{e f_1} u_{e f_2}}_{\delta_{f_1 f_2} u_{e f_1}}
  =
  \sum_{\substack{f \in E \\ s(f) = w}} u_{e f}
  =
  u_{vw}
\end{align*}
and
\[
    \sum_{w \in V} u_{vw}
    = 
    \sum_{w \in V} \sum_{\substack{f \in E \\ s(f) = w}} u_{e f}
    =
    \sum_{f \in E} \underbrace{\sum_{\substack{w \in V \\ s(f) = w}} u_{e f}}_{u_{e f}}
    =
    \sum_{f \in E} u_{e f}
    =
    1.
\]
To show that the rows of $u_V$ sum to $1$, 
we use an argument from the proof of~\cite[Theorem 3.1]{wang98}.
By our previous computation, we have  
\[
  {(u_V u_V^*)}_{vw}
  = \sum_{x \in V} \underbrace{u_{vx} u_{wx}^*}_{\delta_{vw} u_{vx}} = \delta_{vw} \sum_{x \in V} u_{vx} = \delta_{vw}
  \quad \forall v, w \in V.
\]
Hence, $u_V$ is right-invertible with $u_V u_V^* = 1$.
If we show that $u_V$ is a representation of $\AutGHBichon(\Gamma)$, i.e.
\[
  \Delta(u_{vw}) = \sum_{x \in V} u_{vx} \otimes u_{xw}
  \quad 
  \forall v, w \in V,
\]
then~\cite[Proposition 3.2]{woronowicz87} implies that 
$u_V$ is also left-invertibe, i.e.\ $u_V^* u_V = 1$.
Thus, 
\[
  1 = {(u_V^* u_V)}_{vv}
  = \sum_{w \in V} u_{wv}^* u_{wv} 
  = \sum_{w \in V} u_{wv}
  \quad 
  \forall v \in V.
\]
Therefore, it remains to show that $u_V$ is a representation of $\AutGHBichon(\Gamma)$.
Again, let $v, w \in V$ and assume without loss of generality that there exists an edge $e \in E$ with 
$v = s(e)$. Then 

\begin{align*}
  \Delta(u_{vw})
  &= \sum_{\substack{f \in E \\ s(f) = w}} \Delta(u_{ef})
  = \sum_{\substack{f \in E \\ s(f) = w}} \sum_{g \in E} u_{eg} \otimes u_{gf}
  = \sum_{g \in E} u_{eg} \otimes \bigg( \sum_{\substack{f \in E \\ s(f) = w}}  u_{gf} \bigg) \\
  &= \sum_{g \in E} u_{eg} \otimes u_{s(g) w}
\end{align*}
and on the other hand
\begin{align*}
  \sum_{x \in V} u_{vx} \otimes u_{xw}
  = \sum_{x \in V} \sum_{\substack{f \in E \\ s(f) = x}} u_{e f} \otimes u_{xw}
  = \sum_{f \in E} \sum_{\substack{x \in V \\ s(f) = x}} u_{e f} \otimes u_{xw}
  = \sum_{\substack{f \in E}} u_{e f} \otimes u_{s(f) w}.
\end{align*}
Thus,
\[
  \Delta(u_{vw}) = \sum_{x \in V} u_{vx} \otimes u_{xw}
  \quad 
  \forall v, w \in V.
\]
\end{proof}

Using the previous lemmas, we can now show that our quantum
automorphism group agrees with the quantum automorphism group 
of Goswami-Hossain for multigraphs without isolated vertices. 

\begin{theorem}\label{thm:multigraph-GH}
Let $\Gamma$ be a multigraph without isolated vertices. Then 
\[
  \Aut^+(\Gamma) = \AutGHBichon(\Gamma).
\]
\end{theorem}
\begin{proof}
Let $\Gamma := (V, E)$. Denote with $u$ the fundamental representation of $\Aut^+(\Gamma)$
and with $\widehat{u}$ the fundamental representation of $\AutGHBichon(\Gamma)$.
Further, define the elements
\[
  \widehat{u}_{vw} := \begin{cases}
    \displaystyle \sum_{\substack{f \in E \\ s(f) = w}} \widehat{u}_{e f} & \text{if there exists $e \in E$ with $s(e) = v$}, \\
    \displaystyle \sum_{\substack{f \in E \\ r(f) = w}} \widehat{u}_{e f} & \text{if there exists $e \in E$ with $r(e) = v$},
  \end{cases}
\]
for all $v, w \in V$ as in \Cref{def:GH-uV}.
We begin by constructing a unital $*$-homomorphism 
\begin{alignat*}{3}
  \varphi \colon C(\Aut^+(\Gamma)) &\to C(\AutGHBichon(\Gamma)), && \\
  u_{vw} & \mapsto \widehat{u}_{vw} & \quad & \forall v, w \in V, \\
  u_{ef} & \mapsto \widehat{u}_{ef} & \quad & \forall e, f \in E 
\end{alignat*}
using the universal property of $C(\Aut^+(\Gamma))$. Hence, we have to 
show that the matrices
\[
\widehat{u}_V := {(\widehat{u}_{vw})}_{v,w\in V}, \qquad
\widehat{u}_E := {(\widehat{u}_{ef})}_{e,f\in E}
\]
satisfy the relations of \Cref{def:hypergraph-qaut}.
The matrix $\widehat{u}_E$ is a magic unitary 
by Relation~\ref{def:qaut-GH-1} of \Cref{def:qaut-GH}
and $\widehat{u}_V$ is a magic unitary by \Cref{lem:GH-uV-magic}.
Next, consider the relation $A_s \widehat{u}_E = \widehat{u}_V A_s$, 
which is equivalent to $A_s^* \widehat{u}_V = \widehat{u}_E A_s^*$ by \Cref{prop:star-intertwiner}.
However, this relation follows directly, since
\[
    {(A_s^* \widehat{u}_V)}_{ev}
    =
    \sum_{w \in V} {(A_s^*)}_{ew} \widehat{u}_{wv}
    =
    \widehat{u}_{s(e) v}
    =
    \sum_{\substack{f \in E \\ s(f) = v}} \widehat{u}_{ef}
    =
    \sum_{f \in E} \widehat{u}_{ef} {(A_s^*)}_{fv}
    =
    {(\widehat{u}_E A^*_s)}_{ev}
\]
for all $e \in E$ and $v \in V$.
Similarly, one shows $A_r \widehat{u}_E = \widehat{u}_V A_r$,
such that the $*$-homomorphism $\varphi$ exists.

Next, we construct the inverse $*$-homomorphism
\[
  \psi \colon C(\AutGHBichon(\Gamma)) \to C(\Aut^+(\Gamma)),
  \quad 
  \widehat{u}_{ef} \mapsto u_{ef}
  \quad 
  \forall e, f \in E
\]
using the universal property of $C(\AutGHBichon(\Gamma))$.
Thus, we have to show that the entries of $u_E$ satisfy the relations in 
\Cref{def:qaut-GH}. By \Cref{def:hypergraph-qaut}, $u_E$ is a magic unitary.
Further, we can
use \Cref{lem:multigraph-intertwiner-rel} to compute 
\begin{gather*}
  \sum_{\substack{f \in E \\ s(f) = v}} u_{e_1 f}
  =
  u_{s(e_1) v}
  =
  u_{s(e_2) v}
  =
  \sum_{\substack{f \in E \\ s(f) = v}} u_{e_2 f},
  \\
  \sum_{\substack{f \in E \\ r(f) = v}} u_{e_1 f}
  =
  u_{r(e_1) v}
  =
  u_{r(e_2) v}
  =
  \sum_{\substack{f \in E \\ r(f) = v}} u_{e_2 f}
\end{gather*}
for all $v \in V$ and $e_1, e_2 \in E$ with 
$s(e_1) = s(e_2)$ or $r(e_1) = r(e_2)$ respectively. 
Thus, Relation~\ref{def:qaut-GH-2} is satisfied.
Next, consider the Relation~\ref{def:qaut-GH-3} and let $e, f \in E$. If $s(e)$ is neither
a source or sink and $s(f)$ is a source, then $N_{r}(s(e)) > 0$ and $N_{r}(s(f)) = 0$
in the notation of \Cref{prop:N-intertwiner}. Therefore, \Cref{prop:N-intertwiner} implies that 
$u_{s(e)s(f)} = 0$, such that 
\[
  0 = u_{s(e)s(f)} u_{ef} = \sum_{\substack{ g \in E \\ s(g) = s(f)}} \underbrace{u_{eg} u_{ef}}_{\delta_{gf} u_{ef}} = u_{ef}.
\]
Similarly, one shows $u_{s(e)s(f)} = 0$ if $r(f)$ is a sink. Hence, Relation~\ref{def:qaut-GH-3} is satisfied.
Finally, \Cref{lem:multigraph-intertwiner-rel} implies
\[
  \sum_{\substack{f \in E \\ s(f) = v}} u_{e_1 f}
  =
  u_{s(e_1) v}
  =
  u_{r(e_2) v}
  =
  \sum_{\substack{f \in E \\ r(f) = v}} u_{e_2 f}
\]
for all $v \in V$ and $e_1, e_2 \in E$ with $s(e_1) = r(e_2)$.
Thus, Relation~\ref{def:qaut-GH-4} holds and the $*$-homomorphism $\varphi$ exist.

The $*$-homomorphisms $\varphi$ and $\phi$ are indeed inverse, since 
\begin{align*}
  u_{ef} \ & \longleftrightarrow \ \widehat{u}_{ef} \quad \forall e, f \in E, \\
  u_{vw} \ & \longleftrightarrow \ \widehat{u}_{vw} \quad \forall v, w \in V 
\end{align*}
by \Cref{lem:multigraph-intertwiner-rel} and \Cref{def:GH-uV}.
Further, $\psi$ is a morphism of compact quantum groups, because
\[
  \Delta(\psi(\widehat{u}_{ef}))
  = \sum_{g \in E} u_{eg} \otimes u_{gf}
  = \sum_{g \in E} \psi(\widehat{u}_{eg}) \otimes \psi(\widehat{u}_{gf})
  = (\psi \otimes \psi)(\Delta(\widehat{u}_{ef}))
\]
for all $e, f \in E$. Hence, $\Aut^+(\Gamma) = \AutGHBichon(\Gamma)$.
\end{proof}

\section{\texorpdfstring{Action on hypergraph $C^*$-algebras}{Action on hypergraph C*-algebras}}\label{sec:action}

\noindent In~\cite{schmidt18}, Schmidt-Weber showed that the quantum 
automorphism group $\AutBanica(\Gamma)$ of a directed graph $\Gamma$
acts maximally on the graph $C^*$-algebra $C^*(\Gamma)$.
In our notation, this action is given by 
\begin{alignat*}{3}
  \alpha(p_v) &= \sum_{w \in V} p_w \otimes u_{wv}
  \quad&&
  \forall v \in V, \\
  \alpha(s_{(v_1, w_1)}) &= \sum_{(v_2, w_2) \in E} s_{(v_2, w_2)} \otimes u_{v_2 v_1} u_{w_2 w_1}
  \quad&&
  \forall (v_1, w_1) \in E,
\end{alignat*}
where $\Gamma:= (V, E)$ is a direct graph and we flipped both tensor legs to agree with the notation in~\cite{wang98}.

In the following, we generalize this result to hypergraphs by 
showing that our quantum automorphism group $\Aut^+(\Gamma)$ of a hypergraph $\Gamma$
acts faithfully on the hypergraph $C^*$-algebra $C^*(\Gamma)$ from \Cref{def:hypergraph-Cstar-alg}.
Here, the action is given by 
\[
  \alpha(p_v) = \sum_{w \in V} p_w \otimes u_{wv}
  \quad
  \forall v \in V,
  \qquad
  \alpha(s_e) = \sum_{f \in E} s_f \otimes u_{fe}
  \quad
  \forall e \in E,
\]
Further, we show that this action is maximal when also considering a dual action
\[
  \alpha'(p_e) = \sum_{f \in E} p_f \otimes u_{fe}
  \quad
  \forall e \in E,
  \qquad
  \alpha'(s_v) = \sum_{w \in V} s_w \otimes u_{wv}
  \quad
  \forall v \in V
\]
on $C^*(\Gamma')$, where $\Gamma' := {(\Gamma^*)}^{\op}$ is obtained by flipping the source and range maps
in the dual hypergraph $\Gamma^*$.

\subsection{Existence of the action}

We begin by showing the existence of the action of $\Aut^+(\Gamma)$
on the hypergraph $C^*$-algebra $C^*(\Gamma)$, where the 
main step will be the construction of the underlying
$*$-homomorphism
\[
  \alpha \colon C^*(\Gamma) \to C^*(\Gamma) \otimes C(\Aut^+(\Gamma)).
\]
For this construction, we first recall a simple fact 
about $C^*$-algebras, which allows us to compare
tensor products by positive elements.

\begin{proposition}\label{prop:tensor-compare}
Let $\A$ be a $C^*$-algebra and $x, y, z \in \A$. 
If $x \leq y$ and $z \geq 0$, then $x \otimes z \leq y \otimes z$.
\end{proposition}
\begin{proof}
Let $x, y, z \in \A$ with $x \leq y$ and $z \geq 0$. Then there exist $a, b \in \A$ such that 
$y - x = a^*a$ and $z = b^*b$. Thus,
\[
  y \otimes z - x \otimes z = (y - x) \otimes z = a^*a \otimes b^*b = {(a \otimes b)}^* (a \otimes b) \geq 0,
\]
such that $x \otimes z \leq y \otimes z$.
\end{proof}

Next, we construct the $*$-homomorphism $\alpha$
in the following lemma and then show in \Cref{thm:action-existence} that it defines indeed an 
action of $\Aut^+(\Gamma)$ on $C^*(\Gamma)$.

\begin{lemma}\label{lemm:action-alpha-existence}
Let $\Gamma := (V, E)$ be a hypergraph and denote with $u$ the fundamental
representation of $\Aut^+(\Gamma)$. Then
there exists a unital $*$-homomorphism 
\[
  \alpha \colon C^*(\Gamma) \to C^*(\Gamma) \otimes C(\Aut^+(\Gamma))
\] 
with 
\[
  \alpha(p_v) = \sum_{w \in V} p_w \otimes u_{wv}
  \quad
  \forall v \in V,
  \qquad
  \alpha(s_e) = \sum_{f \in E} s_f \otimes u_{fe}
  \quad
  \forall e \in E.
\]
\end{lemma}
\begin{proof}
We use the universal property of $C^*(\Gamma)$ 
to construct the map $\alpha$ from the statement.
Thus, we have to show that $\alpha(p_v)$ are orthogonal projections, 
$\alpha(s_e)$ are partial isometries and both satisfy the relations
from \Cref{def:hypergraph-Cstar-alg}.  
Recall the magic unitary relations  
of $u_V$ and $u_E$ and the intertwiner relations of $A_s$ and $A_r$ from \Cref{rem:hypergraph-qaut-rels}.
Then, we can show that $\alpha(p_v)$ are orthogonal projections by computing
\begin{align*}
  \alpha(p_{v_1}) \alpha(p_{v_2}) 
  &= \sum_{w_1, w_2 \in V} \underbrace{p_{w_1} p_{w_2}}_{\delta_{w_1 w_2} p_{w_1}} \otimes u_{w_1 v_1} u_{w_2 v_2} \\
  &= \sum_{w \in V} p_w \otimes \underbrace{u_{w v_1} u_{w v_2}}_{\delta_{v_1 v_2} u_{w v_1}} 
  = \delta_{v_1 v_2} \sum_{w \in V} p_w \otimes u_{w v_1} 
  = \delta_{v_1 v_2} \, \alpha(p_{v_1})
\end{align*}
for all $v_1, v_2 \in V$ and 
\begin{align*}
  {\alpha(p_v)}^*
  = \sum_{w \in V} p_w^* \otimes u_{w v}^*  
  = \sum_{w \in V} p_w \otimes u_{w v} 
  = \alpha(p_v)
\end{align*}
for all $v \in V$. Similarly, we show that $\alpha(s_e)$ are partial isometries by computing
\begin{align*}
  \alpha(s_e){\alpha(s_e)}^*\alpha(s_e)
  &= \sum_{f_1, f_2, f_3 \in E} s_{f_1} s_{f_2} s_{f_3} \otimes \underbrace{u_{f_1 e} u_{f_2 e}^* u_{f_3 e}}_{\delta_{f_1 f_2} \delta_{f_1 f_3} u_{f_1 e}} \\
  &= \sum_{f \in E} \underbrace{s_f s_f^* s_f}_{s_f} \otimes u_{f e}
  = \sum_{f \in E} s_f \otimes u_{f e} 
  = \alpha(s_e)
\end{align*}
for all $e \in E$. Next, consider Relation~\ref{def:hypergraph-Cstar-alg-R1} from \Cref{def:hypergraph-Cstar-alg}, which states that
\begin{align*}
  s_e^* s_f = \delta_{ef} \sum_{\substack{v \in V \\ v \in r(e)}} p_v \quad \forall e,f \in E.
\end{align*} 
When applying $\alpha$ to the left side, we obtain
\begin{align*}
  {\alpha(s_e)}^* \alpha(s_f)
  &= \sum_{g_1, g_2 \in E} \underbrace{s_{g_1}^* s_{g_2}}_{\text{Rel.\ \ref{def:hypergraph-Cstar-alg-R1}}} \otimes u_{g_1 e}^* u_{g_2 f} \\
  &= \sum_{g \in E} \sum_{\substack{v \in V \\ v \in r(g)}} p_v \otimes \underbrace{u_{g e} u_{g f}}_{\delta_{e f} u_{g e}}
  = \delta_{ef} \sum_{v \in V} \sum_{\substack{g \in E \\ v \in r(g)}} p_v \otimes u_{g e}.
\end{align*}
Using \Cref{rem:hypergraph-qaut-rels}, we have
\[
  \sum_{\substack{g \in E \\ v \in r(g)}} u_{ge}
  =
  \sum_{\substack{w \in V \\ w \in r(e)}} u_{vw}
  \quad 
  \forall v \in V, \, e \in E,
\]
which can be used to further rewrite ${\alpha(s_e)}^* \alpha(s_f)$ to
\begin{align*}
  {\alpha(s_e)}^* \alpha(s_f) 
  &= \delta_{e f} \sum_{v \in V} \sum_{\substack{w \in V \\ w \in r(e)}} p_v \otimes u_{v w} 
  = \delta_{e f} \sum_{\substack{w \in V \\ w \in r(e)}} \sum_{v \in V} p_v \otimes u_{v w} \\
  &= \delta_{e f} \sum_{\substack{w \in V \\ w \in r(e)}} \alpha(p_w).
\end{align*}
Thus, Relation~\ref{def:hypergraph-Cstar-alg-R1} is satisfied.
Next, consider Relation~\ref{def:hypergraph-Cstar-alg-R2}, which is given by
\[
  s_e s_e^* \leq \sum_{\substack{v \in V \\ v \in s(e)}} p_v \quad \forall e \in E.
\]
As before, we apply $\alpha$ to the left side and compute
\[
  \alpha(s_e) {\alpha(s_e)}^*
  = \sum_{f_1, f_2 \in E} s_{f_1} s_{f_2}^* \otimes \underbrace{u_{f_1 e} u_{f_2 e}^*}_{\delta_{f_1 f_2} u_{f_1 e}}  \\
  = \sum_{f \in E} s_{f} s_{f}^* \otimes u_{f e}. 
\]
Since each $u_{fe} \geq 0$, we can use \Cref{prop:tensor-compare} with Relation~\ref{def:hypergraph-Cstar-alg-R2}
to obtain
\[
  \alpha(s_e) {\alpha(s_e)}^*
  \leq 
  \sum_{f \in E} \bigg( \sum_{\substack{v \in V \\ v \in s(f)}} p_v \bigg) \otimes u_{f e}
  =
  \sum_{v \in V } \sum_{\substack{f \in E \\ v \in s(f)}} p_v \otimes u_{f e}.
\]
By \Cref{rem:hypergraph-qaut-rels}, we have 
\[
  \sum_{\substack{f \in E \\ v \in s(f)}} u_{f e} 
  =
  \sum_{\substack{w \in V \\ w \in s(e)}} u_{w v}
  \quad 
  \forall v \in V, \, e \in E,
\]
such that 
\[
  \alpha(s_e) {\alpha(s_e)}^* 
  \leq 
  \sum_{v \in V } \sum_{\substack{w \in V \\ w \in s(e)}} p_v \otimes u_{vw}
  =
  \sum_{\substack{w \in V \\ w \in s(e)}} \sum_{v \in V } p_v \otimes u_{vw} 
  =
  \sum_{\substack{w \in V \\ w \in s(e)}} \alpha(p_w).
\]
Hence, Relation~\ref{def:hypergraph-Cstar-alg-R2} is satisfy.
Finally, consider Relation~\ref{def:hypergraph-Cstar-alg-R3}, which states that 
\[
  p_v \leq \sum_{\substack{e \in E \\ v \in s(e)}} s_e s_e^*,
\]
for all $v \in V$ which are not a sink. Let $v \in V$ be not a sink. 
By \Cref{prop:N-intertwiner}, we have $u_{vw} = 0$ for all $w \in V$ which are a sink, since 
$N_s(v) > 0$ if $v$ is not a sink and $N_s(w) = 0$ if $w$ is a sink.
Therefore,
\[
  \alpha(p_v) 
  = \sum_{w \in V} p_w \otimes u_{w v}
  = \sum_{\substack{w \in V \\ \text{$w$ $\neq$ sink}}} p_w \otimes u_{w v}.
\]
Applying \Cref{prop:tensor-compare} with Relation~\ref{def:hypergraph-Cstar-alg-R3} yields
\[
  \alpha(p_v) 
  \leq \sum_{\substack{w \in V \\ \text{$w$ $\neq$ sink}}} \bigg( \sum_{\substack{e \in E \\ w \in s(e)}} s_e s_e^* \bigg) \otimes u_{w v}
  = \sum_{e \in E} \sum_{\substack{w \in V \\ w \in s(e)}} s_e s_e^* \otimes u_{w v},
\]
where we used again the fact that $u_{v w} = 0$ if $w$ is a sink.
On the other hand, we have
\begin{align*}
  \sum_{\substack{f \in E \\ v \in s(f)}}
  \alpha(s_f) {\alpha(s_f)}^* 
  =
  \sum_{e_1, e_2 \in E} \sum_{\substack{f \in E \\ v \in s(f)}}
  s_{e_1} s_{e_2}^* \otimes \underbrace{u_{e_1 f} u_{e_2 f}^*}_{\delta_{e_1 e_2} u_{e_1 f}}
  = 
  \sum_{e \in E} \sum_{\substack{f \in E \\ v \in s(f)}}
  s_e s_e^* \otimes u_{e f}.
\end{align*}
By \Cref{rem:hypergraph-qaut-rels}, we have 
\[
  \sum_{\substack{w \in V \\ w \in s(e)}}  u_{w v}
  =
  \sum_{\substack{f \in E \\ v \in s(f)}} u_{e f}
  \quad 
  \forall v \in V, \, e \in E,
\]
which implies 
\begin{align*}
  \alpha(p_v) 
  \leq 
  \sum_{e \in E} \sum_{\substack{w \in V \\ w \in s(e)}} 
   s_e s_e^* \otimes u_{w v}
  =
  \sum_{e \in E} \sum_{\substack{f \in E \\ v \in s(f)}}
  s_e s_e^* \otimes u_{e f}
  =  
  \sum_{\substack{f \in E \\ v \in s(f)}}
  \alpha(s_f) {\alpha(s_f)}^*.
\end{align*}
\end{proof}

Next, we show that the previous $*$-homomorphism $\alpha$ defines 
a faithful action in the sense of \Cref{def:qgroup-action} and \Cref{def:qgroup-faithful}.

\begin{theorem}\label{thm:action-existence}
Let $\Gamma$ be a hypergraph. Then
$\Aut^+(\Gamma)$ acts faithfully on $C^*(\Gamma)$ via the map $\alpha$ from \Cref{lemm:action-alpha-existence}.
\end{theorem}
\begin{proof}
Let $\Gamma := (V, E)$ and define $\B \subseteq C^*(\Gamma)$ as
the $*$-subalgebra generated by $p_v$ for all $v \in V$ and $p_e$
for all $e \in E$. Then $\B$ is dense in $C^*(\Gamma)$ and 
\[
  \alpha(\B) \subseteq \B \otimes \OO(\Aut^+(\Gamma))
\] 
by the definition of $\alpha$. Next, let $v \in V$. Then
\[
  (\alpha \otimes \id)(\alpha(p_v)) 
  =
  (\id \otimes \Delta)(\alpha(p_v)),
\]
since
\[
  \sum_{w \in V} \alpha(p_w) \otimes u_{wv} 
  = \sum_{w_1, w_2 \in V} p_{w_2} \otimes u_{w_2 w_1} \otimes u_{w_1 v} 
  = \sum_{w_2 \in V}  p_{w_2} \otimes \Delta(u_{w_2 v}).
\]
Further, we compute
\[
  (\id \otimes \varepsilon)(\alpha(p_v))
  = (\id \otimes \varepsilon)\bigg(\sum_{w \in V} p_{w} \otimes u_{wv} \bigg) 
  = \sum_{w \in V} p_w \cdot \delta_{wv}
  = p_v.
\]
The previous computations also show that 
\begin{alignat*}{3}
  (\alpha \otimes \id)(\alpha(s_e)) &= 
  (\id \otimes \Delta)(\alpha(s_e)) 
  & \qquad &
  \forall e \in E, \\
  (\id \otimes \varepsilon)(\alpha(s_e)) &=
  s_e
  & \qquad &
  \forall e \in E
\end{alignat*}
by replacing $p_v$ with $s_e$. Hence,
\begin{align*}
  (\id \otimes \alpha) \circ \alpha &= (\Delta \otimes \id) \circ \alpha, \\
  (\varepsilon \otimes \id) \circ \alpha|_{\B} &= \id.
\end{align*}
Thus, $\alpha$ defines an action of $\Aut^+(\Gamma)$ on $C^*(\Gamma)$.
To show that $\alpha$ is faithful, assume there exists a quotient $G$ of $\Aut^+(\Gamma)$
such that $\alpha|_{C(G)}$ is also an action on $C^*(\Gamma)$.
Then 
\begin{gather*}
  \alpha|_{C(G)}(p_v) = \sum_{w \in V} p_w \otimes u_{wv} \in C^*(\Gamma) \otimes C(G) 
  \qquad 
  \forall v \in V, \\
  \alpha|_{C(G)}(s_e) = \sum_{f \in E} s_f \otimes u_{fe} \in C^*(\Gamma) \otimes C(G) 
  \qquad 
  \forall e \in E.
\end{gather*}
The representation of $\alpha|_{C(G)}(p_v)$ with respect to $p_v$ is unique since the $p_v$ are linearly independent 
as orthogonal projections. Thus, $u_{wv} \in C(G)$ for all $v, w \in V$. 
Similarly, $u_{ef} \in C(G)$ for all $e, f \in E$, since the $s_e$ are 
linearly independent as partial isometries with orthogonal ranges.
Hence, $C(G) = C(\Aut^+(\Gamma))$, which shows that $\alpha$ is faithful.
\end{proof}

Although our action appears to be different from the action of Schmidt-Weber, 
the following remark shows that it reduces to the action of Schmidt-Weber in the case of classical directed graphs. 
In particular, it justifies the special form of the action retrospectively, since the action does not appear to be canonical from the point of view of classical directed graphs.

\begin{remark}\label{rem:action-generalize}
  Let $\Gamma:= (V, E)$ be a directed graph. By the proof of \Cref{thm:directed-graph-bichon},
  we have $\Aut^+(\Gamma) = \AutBichon(\Gamma)$ via 
  \begin{alignat*}{3}
    u_{vw} & \ \mapsto \ \widehat{u}_{vw} && \quad \forall v \in V, \\
    u_{(v_1, v_2)(w_1, w_2)} & \ \mapsto \ \widehat{u}_{v_1 w_1} \widehat{u}_{v_2 w_2} &&\quad \forall (v_1,v_2), (w_1, w_2) \in E,
  \end{alignat*}
  where $u$ denotes the fundamental representation of $\Aut^+(\Gamma)$
  and $\widehat{u}$ denotes the fundamental representation of $\AutBichon(\Gamma)$.
  Under this isomorphism, the action $\alpha$ from \Cref{thm:action-existence} has the 
  form 
  \begin{alignat*}{3}
    \alpha(p_v) &= \sum_{w \in V} p_w \otimes \widehat{u}_{wv}
    \quad&&
    \forall v \in V, \\
    \alpha(s_{(v_1, w_1)}) &= \sum_{(v_2, w_2) \in E} s_{(v_2, w_2)} \otimes \widehat{u}_{v_2 v_1} \widehat{u}_{w_2 w_1}
    \quad&&
    \forall (v_1, w_1) \in E,
  \end{alignat*}
  which is the same form as in~\cite{schmidt18}.
  Hence, we obtain the action of Schmidt-Weber for $\AutBichon(\Gamma)$ as a special case.
  Note that it was already shown in~\cite{joardar18} that $\AutBichon(\Gamma)$ acts on 
  $C^*(\Gamma)$ via this action.
\end{remark}

\subsection{Maximality of the action}

In~\cite{schmidt18}, Schmidt-Weber showed that their action is maximal in the sense that 
$\AutBanica(\Gamma)$ is the largest quantum groups which acts on the graph $C^*$-algebra $C^*(\Gamma)$
via the map $\alpha$ decribed before.

However, if $\Gamma$ is a classical directed graph, then our quantum automorphism
group agress with $\AutBichon(\Gamma)$ by \Cref*{thm:directed-graph-bichon} and 
our action $\alpha$ in \Cref{thm:action-existence} agrees with the action of Schmidt-Weber by
\Cref{rem:action-generalize}.
Since $\AutBichon(\Gamma) \subset \AutBanica(\Gamma)$ for some direct graphs, our action on hypergraph $C^*$-algebras is not maximal,
even in the case of directed graphs.

In the following, we show that we can obtain maximality of our action by including an additional assumption. To 
formulate this assumption, we first introduce some notation. 

\begin{definition}\label{def:gamma-prime}
Let $\Gamma:= (V, E)$ be a hypergraph. Then 
define $\Gamma' := (E, V)$ with 
\[
  s'(e) := \{ v \in V ~|~ v \in r(e) \},
  \quad
  r'(e) := \{ v \in V ~|~ v \in s(e) \}
  \quad 
  \forall e \in E.
\]
Note that $\Gamma' = {(\Gamma^{\op})}^* = {(\Gamma^*)}^{\op}$, where 
$\Gamma^{\op}$ and $\Gamma^*$ are the opposite and dual hypergraph
from \Cref{def:opposite-hypergraph} and \Cref{def:dual-hypergraph}.
\end{definition}

Using the results in \Cref{sec:opp-dual-hypergraphs}, it follows directly that 
$\Aut^+(\Gamma)$ does not only act on $C^*(\Gamma)$ but also on $C^*(\Gamma')$ 
in a natural way.

\begin{proposition}\label{lem:gamma-prime-action}
Let $\Gamma:= (V, E)$ be a hypergraph.
Then $\Aut^+(\Gamma)$ acts faithfully on $C^*(\Gamma')$ via
\[
  \alpha \colon C^*(\Gamma') \to C^*(\Gamma') \otimes C(\Aut^+(\Gamma))
\] 
with 
\[
  \alpha(p_e) = \sum_{f \in E} p_f \otimes u_{fe}
  \quad
  \forall e \in E,
  \qquad
  \alpha(s_v) = \sum_{w \in V} s_w \otimes u_{wv}
  \quad
  \forall v \in V.
\]
\end{proposition}
\begin{proof}
Denote with $u$ the fundamental
representation of $\Aut^+(\Gamma)$ 
and with $\widehat{u}$ the fundamental representation 
$\Aut^+(\Gamma)$. By \Cref{prop:opposite-isom} and \Cref{prop:dual-isom}
we have 
\[
  \Aut^+(\Gamma') = \Aut^+\left({(\Gamma^*)}^{\op}\right)
  = \Aut^+(\Gamma^*) = \Aut^+(\Gamma)
\]
via 
\[
  \widehat{u}_V \ \longleftrightarrow \ u_E,
  \qquad
  \widehat{u}_E \ \longleftrightarrow \ u_V.
\]
The statement then follows by applying this isomorphism
to the action of $\Aut^+(\Gamma')$ on $C^*(\Gamma')$ from \Cref{thm:action-existence}.
\end{proof}

The main observation for proving maximality
is the following lemma, which
shows that at least
some of the relations in \Cref{def:hypergraph-qaut}
can be recovered from the action on $C^*(\Gamma)$.

\begin{lemma}\label{lem:partial-hypergraph-rel}
Let $\Gamma := (V, E)$ be a hypergraph and $G$ be a compact
matrix quantum group which acts on $C^*(\Gamma)$
via $\alpha \colon C^*(\Gamma) \to C^*(\Gamma) \otimes C(G)$ with 
\[
  \alpha(p_v) = \sum_{w \in V} p_w \otimes u_{wv}
  \quad
  \forall v \in V,
  \qquad
  \alpha(s_e) = \sum_{f \in E} s_f \otimes u_{fe}
  \quad
  \forall e \in E,
\]
for some elements $u_{vw} \in C(G)$ for all $v, w\in V$ and $u_{ef} \in C(G)$ for all $e, f \in E$. Then 
$u_V := {(u_{vw})}_{v,w \in V}$ is a magic unitary.
If additionally $u_E := {(u_{ef})}_{e,f\in E}$ is a magic unitary, then 
\[
  A_r u_E = u_V A_r.
\]
\end{lemma}
\begin{proof}
The proof that $u_V$ is a magic unitary is contained in the proof of~\cite[Theorem 3.1]{wang98}, since 
the elements $p_v$ are orthogonal projections which sum to $1$.
For the second part of the statement, assume that $u_E$ is also a magic unitary and consider
the relation
\[
  s_e^* s_e = \sum_{\substack{v \in V \\ v \in r(e)}} p_v \quad \forall e \in E
\]
from \Cref{def:hypergraph-Cstar-alg}. By applying $\alpha$ to the left side and using the relation, we obtain 
\[
  \alpha(s_e^* s_e)
  =
  \sum_{f_1, f_2 \in E} s_{f_1}^* s_{f_2} \otimes \underbrace{u_{f_1 e}^* u_{f_2 e}}_{\delta_{f_1 f_2} u_{f_1 e}}
  =
  \sum_{f \in E} s_f^* s_f \otimes u_{f e}
  =
  \sum_{f \in E} \sum_{\substack{v \in V \\ v \in r(f)}} p_v \otimes u_{f e},
\]
which can be rewritten as
\[
  \alpha(s_e^* s_e)  
  =
  \sum_{v \in V}  \sum_{f \in E} {(A_r)}_{vf} \, p_v \otimes u_{f e}
  =
  \sum_{v \in V} p_v \otimes \bigg(\sum_{f \in E} {(A_r)}_{v f} \, u_{f e} \bigg).
\]
Similarly, by applying $\alpha$ to the right side of the 
original equation, we obtain
\begin{align*}
  \sum_{\substack{v \in V \\ v \in r(e)}} \alpha(p_v)
  &=
  \sum_{\substack{v \in V \\ v \in r(e)}}
  \sum_{w \in V} p_w \otimes u_{w v} \\
  &=
  \sum_{v, w \in V}
  {(A_r)}_{v e} \, p_w \otimes u_{w v}
  =
  \sum_{w \in V}  p_w \otimes \bigg( \sum_{v \in V}  u_{w v} {(A_r)}_{v e} \bigg).
\end{align*}
Thus,
\[
  {(A_r u_E)}_{v e} 
  = \sum_{f \in E} {(A_r)}_{v f} \, u_{f e}
  = \sum_{w \in V} u_{w v} {(A_r)}_{v e}
  = {(A_r u_V)}_{v e} 
  \quad
  \forall v \in V
\]
by using the fact that the elements $p_v$ are linearly independent as orthogonal projections, such that we can compare the terms in the previous sums. This shows $A_r u_E = u_V  A_r$.
\end{proof}

By combining the previous lemma with the actions on $C^*(\Gamma)$
and $C^*(\Gamma')$, we can now show that $\Aut^+(\Gamma)$ is the largest
quantum group which acts both on $C^*(\Gamma)$ and $C^*(\Gamma')$
in the sense of \Cref{thm:action-existence} and \Cref{lem:gamma-prime-action}.

\begin{theorem}\label{thm:action-maximal}
Let $\Gamma := (V, E)$ be a hypergraph 
and $G$ be a compact matrix quantum group which
acts faithfully on $C^*(\Gamma)$ and $C^*(\Gamma')$
via
\[
  \alpha_1 \colon C^*(\Gamma) \to C^*(\Gamma) \otimes C(G),
  \qquad
  \alpha_2 \colon C^*(\Gamma') \to C^*(\Gamma') \otimes C(G),
\]
which satisfy
\begin{alignat*}{3}
  \alpha_1(p_v) &= \sum_{w \in V} p_w \otimes u_{wv}
  \quad
  \forall v \in V,
  &\qquad
  \alpha_1(s_e) &= \sum_{f \in E} s_f \otimes u_{fe}
  \quad
  \forall e \in E, \\
  \alpha_2(p_e) &= \sum_{f \in E} p_f \otimes u_{fe}
  \quad
  \forall e \in E,
  &\qquad
  \alpha_2(s_v) &= \sum_{w \in V} s_w \otimes u_{wv}
  \quad
  \forall v \in V
\end{alignat*}
for some elements $u_{vw} \in C(G)$ for all $v, w\in V$ and $u_{ef} \in C(G)$ for all $e, f \in E$. Then $G \subseteq \Aut^+(\Gamma)$.
\end{theorem}
\begin{proof}
By applying the first part of \Cref{lem:partial-hypergraph-rel} to $\alpha_1$ and $\alpha_2$, we obtain
that the matrices $u_V := {(u_{vw})}_{v,w \in V}$
and $u_E := {(u_{ef})}_{e,f\in E}$ are magic unitaries. Then the second
part of \Cref{lem:partial-hypergraph-rel} yields, that 
\[
  A_r u_E = u_V A_r, \quad A_{r'} u_V = u_E A_{r'}.
\]
But $A_{r'} = A_s^*$, such that $A_s^* u_V = u_E A_s^*$, which implies $A_s u_E = u_V A_s$ by \Cref{prop:star-intertwiner}.
This shows that the elements $u_{vw}$ and $u_{ef}$ satisfy the relations
from \Cref{def:hypergraph-qaut}. Hence, by the universal property of $C(\Aut^+(\Gamma))$,
there exists a unital $*$-homomorphism
\[
  \varphi \colon C(\Aut^+(\Gamma)) \to C(G)
\]
which maps the generators of $C(\Aut^+(\Gamma))$ to the entries of $u_V$ and $u_E$.
Next, we show that $\varphi$ is a morphism of compact quantum groups.
Let $w \in V$ and observe that
\begin{align*}
  (\alpha_1 \otimes \id)(\alpha_1(w)) 
  &= 
  \sum_{x \in V} \alpha_1(p_{x}) \otimes u_{x w} \\
  &=
  \sum_{x \in V} \sum_{v \in V} p_{v} \otimes u_{v x} \otimes u_{x w}
  =
  \sum_{v \in V} p_{v} \otimes \bigg(\sum_{x \in V} u_{v x} \otimes u_{x w}\bigg)
\end{align*}
and
\[
  (\id \otimes \Delta)(\alpha_1(w)) 
  =
  \sum_{v \in V} p_{v} \otimes \Delta(u_{v w}).
\]
Because $\alpha_1$ is an action, it satisfies 
$(\alpha_1 \otimes \id) \circ \alpha_1 = (\id \otimes \Delta) \circ \alpha_1$.
Further, all $p_v$ are linearly independent as orthogonal projections, which implies
\[
  \Delta(u_{vw}) = \sum_{w \in V} u_{v x} \otimes u_{x w}
  \quad 
  \forall v, w \in V.
\]
Denote with $\widehat{u}$ the fundamental representation of $\Aut^+(\Gamma)$. Then the previous equation yields
\[
  \Delta(\varphi(\widehat{u}_{v w}))
  = 
  \sum_{x \in V} u_{v x} \otimes u_{x w}
  =
  \sum_{x \in V} \varphi(\widehat{u}_{v x}) \otimes \varphi(\widehat{u}_{x w}),
  =
  (\varphi \otimes \varphi)(\Delta(\widehat{u}_{v w}))
\]
for all $v, w \in V$. Similarly, one shows 
\[
  \Delta(\varphi(\widehat{u}_{e f}))
  =
  (\varphi \otimes \varphi)(\Delta(\widehat{u}_{e f}))
  \quad 
  \forall e, f \in E
\]
by using the action $\alpha_2$.
Thus, $\varphi$ is a morphism of compact quantum groups.
Further, $\varphi$ is surjective, because $\alpha_1$ and $\alpha_2$ are faithful.
Otherwise, the image of $\varphi$ would define a proper quotient quantum group of 
$G$ acting on $C^*(\Gamma)$ and $C^*(\Gamma')$ in the same way, which is impossible.
Therefore, $G \subseteq \Aut^+(\Gamma)$.
\end{proof}

Note that if $\Gamma$ is an undirected
hypergraph, then $\Gamma' = \Gamma^*$ in \Cref{def:gamma-prime}.
In this case, $\Aut^+(\Gamma)$ is the maximal quantum groups 
which acts faithfully on both $C^*(\Gamma)$ and $C^*(\Gamma^*)$ in a 
compatible way.

Further, since $\Aut^+(\Gamma)$ acts maximally, we can view $\Aut^+(\Gamma)$ as the 
quantum symmetry group of $C^*(\Gamma)$ in the sense of \Cref{thm:action-maximal}.

\section{Open questions}\label{sec:open-questions}

In this last section, we present some remaining open questions. First, it would be interesting to 
compute further quantum automorphism groups of concrete hypergraphs.
Consider for example a complete hypergraph in the following sense.

\begin{definition}
Let $V$ be a finite set. Then the \textit{complete hypergraph} on $V$ is given by $\Gamma := (V, \PP(V) \times \PP(V))$ with 
source and range maps
\[
  s(X, Y) = X, \quad r(X, Y) = Y.
\]
\end{definition}

Since $\Gamma$ has no multi-edges, we know that $\Aut^+(\Gamma) \subseteq S_V^+$ by \Cref{corr:subgroup-SV}.
However, it remains open if this inclusion is proper or if $\Aut^+(\Gamma) = S_V^+$ holds.

\begin{question}
What is the quantum automorphism group of a complete hypergraph?
\end{question}

Further, we showed in \Cref{sec:qsym-graphs} that our quantum automorphism
group generalizes the quantum automorphism group of classical graphs by Bichon.
However, it might be possible that each quantum automorphism
group of a hypergraph can be realized as the quantum automorphism group of a possibly larger classical graph. 

\begin{question}
  Let $\Gamma$ be a hypergraph. Can one construct a classical graph $\Gamma'$,
  such that $\Aut^+(\Gamma) = \AutBichon(\Gamma')$? 
\end{question}

It also remains open if there exists a Banica version for the quantum automorphism group of hypergraphs in the following sense.

\begin{question}
  Is there an alternative definition of quantum automorphism groups of hypergraphs that reduces to $\AutBanica(\Gamma)$ in the case of classical directed graphs and 
  which acts on $C^*(\Gamma)$ in a natural way?
\end{question}

Finally, Hahn~\cite{hahn81} characterized hypergraphs for which 
the classical automorphism group of their product is given by the wreath product of their automorphism groups.

\begin{question}
  Can the results of Hahn be generalized to the quantum setting using the free wreath product of Bichon~\cite{bichon04}?
\end{question}

\bibliographystyle{amsplain}
\bibliography{notes}

\providecommand{\bysame}{\leavevmode\hbox to3em{\hrulefill}\thinspace}
\providecommand{\MR}{\relax\ifhmode\unskip\space\fi MR }
\providecommand{\MRhref}[2]{%
  \href{http://www.ams.org/mathscinet-getitem?mr=#1}{#2}
}
\providecommand{\href}[2]{#2}
\begin{thebibliography}{10}

\bibitem{ausiello17}
Giorgio Ausiello and Luigi Laura, \emph{Directed hypergraphs: Introduction and fundamental algorithms - a survey}, Theoretical Computer Science \textbf{658} (2017), 293--306.

\bibitem{banica05}
Teodor Banica, \emph{Quantum automorphism groups of homogeneous graphs}, Journal of Functional Analysis \textbf{224} (2005), no.~2, 243--280.

\bibitem{berge84}
Claude Berge, \emph{Hypergraphs: Combinatorics of finite sets}, North-Holland Mathematical Library, Elsevier Science, 1984.

\bibitem{bichon03}
Julien Bichon, \emph{Quantum automorphism groups of finite graphs}, Proceedings of the American Mathematical Society \textbf{131} (2003), no.~3, 665–673.

\bibitem{bichon04}
\bysame, \emph{Free wreath product by the quantum permutation group}, Algebras and Representation Theory \textbf{7} (2004), 343--362.

\bibitem{blackadar06}
Bruce Blackadar, \emph{Operator algebras. theory of {C}*-algebras and von {N}eumann algebras}, Encyclopaedia of Mathematical Sciences, Springer, 2006.

\bibitem{brannan20}
Michael Brannan, Alexandru Chirvasitu, Kari Eifler, Samuel Harris, Vern Paulsen, Xiaoyu Su, and Mateusz Wasilewski, \emph{Bigalois extensions and the graph isomorphism game}, Communications in Mathematical Physics \textbf{375} (2020), 1777--1809.

\bibitem{brannan22}
Michael Brannan, Kari Eifler, Christian Voigt, and Moritz Weber, \emph{Quantum {C}untz-{K}rieger algebras}, Transactions of the American Mathematical Society Series B \textbf{9} (2022), 782--826.

\bibitem{chassaniol19}
Arthur Chassaniol, \emph{Study of quantum symmetries for vertex-transitive graphs using intertwiner spaces}, arXiv:1904.00455, 2019.

\bibitem{cuntz80}
Joachim Cuntz and Wolfgang Krieger, \emph{A class of {C*}-algebras and topological {M}arkov chains}, Inventiones mathematicae \textbf{56} (1980), 251--268.

\bibitem{bruyn23}
{Josse van} {Dobben de Bruyn}, Prem~Nigam Kar, David~E. Roberson, Simon Schmidt, and Peter Zeman, \emph{Quantum automorphism groups of trees}, arXiv:2311.04891, 2023.

\bibitem{faross23}
Nicolas Faroß and Moritz Weber, \emph{A concrete model for the quantum permutation group on 4 points}, Experimental Mathematics (2024), 1--14.

\bibitem{gallo93}
Giorgio Gallo, Giustino Longo, Stefano Pallottino, and Sang Nguyen, \emph{Directed hypergraphs and applications}, Discrete Applied Mathematics \textbf{42} (1993), no.~2-3, 177--201.

\bibitem{goswami23}
Debashish Goswami and Sk~Asfaq Hossain, \emph{Quantum symmetry in multigraphs}, arXiv:2302.08726v3, 2023.

\bibitem{gromada22}
Daniel Gromada, \emph{Quantum symmetries of {H}adamard matrices}, arXiv:2210.02047, 2022.

\bibitem{hahn81}
Geňa Hahn, \emph{The automorphism group of a product of hypergraphs}, Journal of Combinatorial Theory, Series B \textbf{30} (1981), no.~3, 276--281.

\bibitem{joardar18}
Soumalya Joardar and Arnab Mandal, \emph{Quantum symmetry of graph {C}*-algebras associated with connected graphs}, Infinite Dimensional Analysis, Quantum Probability and Related Topics \textbf{21} (2018), no.~3, 1850019.

\bibitem{levandovskyy22}
Viktor Levandovskyy, Christian Eder, Andreas Steenpass, Simon Schmidt, Julien Schanz, and Moritz Weber, \emph{Existence of quantum symmetries for graphs on up to seven vertices: A computer based approach}, Proceedings of the 2022 International Symposium on Symbolic and Algebraic Computation, 2022, p.~311–318.

\bibitem{lupini20}
Martino Lupini, Laura Mančinska, and David~E. Roberson, \emph{Nonlocal games and quantum permutation groups}, Journal of Functional Analysis \textbf{279} (2020), no.~5, 108592.

\bibitem{nechita23}
Ion Nechita, Simon Schmidt, and Moritz Weber, \emph{Sinkhorn algorithm for quantum permutation groups}, Experimental Mathematics \textbf{32} (2023), no.~1, 156--168.

\bibitem{neshveyev13}
Sergey Neshveyev and Lars Tuset, \emph{Compact quantum groups and their representation categories}, Cours Spécialisés, Société Mathématique de France, 2013.

\bibitem{raeburn05}
Iain Raeburn, \emph{Graph algebras}, CBMS Regional Conference Series in Mathematics, American Mathematical Society, 2005.

\bibitem{schmidt18-2}
Simon Schmidt, \emph{The petersen graph has no quantum symmetry}, Bulletin of the London Mathematical Society \textbf{50} (2018), no.~3, 395--400.

\bibitem{schmidt20-2}
\bysame, \emph{On the quantum symmetry of distance-transitive graphs}, Advances in Mathematics \textbf{368} (2020), 107150.

\bibitem{schmidt20}
\bysame, \emph{Quantum automorphism groups of finite graphs}, Ph.D. thesis, Saarland University, 2020.

\bibitem{schmidt18}
Simon Schmidt and Moritz Weber, \emph{Quantum symmetries of graph {C}*-algebras}, Canadian Mathematical Bulletin \textbf{61} (2018), no.~4, 848--864.

\bibitem{schaefer24}
Björn Schäfer and Moritz Weber, \emph{Nuclearity of hypergraph {C*}-algebras}, arXiv:2405.10044, 2024.

\bibitem{timmermann08}
Thomas Timmermann, \emph{An invitation to quantum groups and duality: From {H}opf algebras to multiplicative unitaries and beyond}, EMS textbooks in mathematics, EMS Press, 2008.

\bibitem{trieb24}
Mirjam Trieb, Moritz Weber, and Dean Zenner, \emph{Hypergraph {C*}-algebras}, arXiv:2405.09214, 2024.

\bibitem{wang95}
Shuzhou Wang, \emph{Free products of compact quantum groups}, Communications in Mathematical Physics \textbf{167} (1995), 671--692.

\bibitem{wang98}
\bysame, \emph{Quantum symmetry groups of finite spaces}, Communications in Mathematical Physics \textbf{195} (1998), 195--211.

\bibitem{weber23}
Moritz Weber, \emph{Quantum permutation matrices}, Complex Analysis and Operator Theory \textbf{17} (2023), no.~37.

\bibitem{woronowicz87}
Stanisław Woronowicz, \emph{Compact matrix pseudogroups}, Communications in Mathematical Physics \textbf{111} (1987), 613--665.

\bibitem{woronowicz91}
\bysame, \emph{A remark on compact matrix quantum groups}, Letters in Mathematical Physics \textbf{21} (1991), 35--39.

\end{thebibliography}

\end{document}